\documentclass{amsart}
\usepackage{latexsym,amsfonts,amsmath,amssymb,bbm}
\usepackage{diagrams}
\usepackage{verbatim}

\diagramstyle[tight,centredisplay]
\newtheorem{theorem}{Theorem}

\newtheorem{corollary}[theorem]{Corollary}

\newtheorem*{questions*}{Questions}
\newtheorem*{mainquestion*}{Main Question} 

\newcommand{\QED}{\end{proof}}

\def\proclaim[#1]{{\bf #1}}
\def\BF#1.{{\bf #1.}}

\newcommand{\Los}{\L o\'s}

\newcommand{\Godel}{G\"odel}

\newcommand{\Levy}{L\'{e}vy}

\newcommand{\X}{{\mathbb X}}

\renewcommand{\P}{{\mathbb P}}

\newcommand{\R}{{\mathbb R}}

\newcommand{\Dbar}{{\overline{D}}}

\newcommand{\Gbar}{{\overline{G}}}
\newcommand{\Mbar}{{\overline{M}}}
\newcommand{\Nbar}{{\overline{N}}}
\newcommand{\Vbar}{{\overline{V}}}

\newcommand{\id}{\mathop{\hbox{\small id}}}
\newcommand{\one}{\mathbbm{1}} 
\newcommand{\surj}{\twoheadrightarrow}
\newcommand{\of}{\subseteq}

\newcommand{\set}[1]{\{\,{#1}\,\}}

\newcommand{\elesub}{\prec}

\newcommand{\ran}{\mathop{\rm ran}}

\newcommand{\Add}{\mathop{\rm Add}}

\newcommand{\Coll}{\mathop{\rm Coll}}
\newcommand{\Ult}{\mathop{\rm Ult}}

\newcommand{\image}{\mathbin{\hbox{\tt\char'42}}}
\newcommand{\plus}{{+}}

\newcommand{\restrict}{\upharpoonright} 
\newcommand{\satisfies}{\models}
\newcommand{\forces}{\Vdash}

\newcommand{\Union}{\bigcup}

\newcommand{\intersect}{\cap}
\newcommand{\Intersect}{\bigcap}

\newcommand{\smalllt}{\mathrel{\mathchoice{\raise2pt\hbox{$\scriptstyle<$}}{\raise1pt\hbox{$\scriptstyle<$}}{\raise0pt\hbox{$\scriptscriptstyle<$}}{\scriptscriptstyle<}}}
\newcommand{\smallleq}{\mathrel{\mathchoice{\raise2pt\hbox{$\scriptstyle\leq$}}{\raise1pt\hbox{$\scriptstyle\leq$}}{\raise1pt\hbox{$\scriptscriptstyle\leq$}}{\scriptscriptstyle\leq}}}
\newcommand{\lt}{\smalllt}
\newcommand{\ltomega}{{{\smalllt}\omega}}
\newcommand{\leqomega}{{{\smallleq}\omega}}
\newcommand{\ltkappa}{{{\smalllt}\kappa}}
\newcommand{\leqkappa}{{{\smallleq}\kappa}}

\newcommand{\boolval}[1]{\mathopen{\lbrack\!\lbrack}\,#1\,\mathclose{\rbrack\!\rbrack}}
\def\[#1]{\boolval{#1}}
\newbox\gnBoxA
\newdimen\gnCornerHgt
\setbox\gnBoxA=\hbox{\tiny$\ulcorner$}
\global\gnCornerHgt=\ht\gnBoxA
\newdimen\gnArgHgt
\def\gcode #1{%
\setbox\gnBoxA=\hbox{$#1$}%
\gnArgHgt=\ht\gnBoxA%
\ifnum     \gnArgHgt<\gnCornerHgt \gnArgHgt=0pt%
\else \advance \gnArgHgt by -\gnCornerHgt%
\fi \raise\gnArgHgt\hbox{\tiny$\ulcorner$} \box\gnBoxA %
\raise\gnArgHgt\hbox{\tiny$\urcorner$}}
\newcommand{\UnderTilde}[1]{{\setbox1=\hbox{$#1$}\baselineskip=0pt\vtop{\hbox{$#1$}\hbox to\wd1{\hfil$\sim$\hfil}}}{}}
\newcommand{\Undertilde}[1]{{\setbox1=\hbox{$#1$}\baselineskip=0pt\vtop{\hbox{$#1$}\hbox to\wd1{\hfil$\scriptstyle\sim$\hfil}}}{}}
\newcommand{\undertilde}[1]{{\setbox1=\hbox{$#1$}\baselineskip=0pt\vtop{\hbox{$#1$}\hbox to\wd1{\hfil$\scriptscriptstyle\sim$\hfil}}}{}}
\newcommand{\UnderdTilde}[1]{{\setbox1=\hbox{$#1$}\baselineskip=0pt\vtop{\hbox{$#1$}\hbox to\wd1{\hfil$\approx$\hfil}}}{}}
\newcommand{\Underdtilde}[1]{{\setbox1=\hbox{$#1$}\baselineskip=0pt\vtop{\hbox{$#1$}\hbox to\wd1{\hfil\scriptsize$\approx$\hfil}}}{}}

\newcommand{\st}{\mid}
\renewcommand{\th}{{\hbox{\scriptsize th}}}

\def\<#1>{\left\langle#1\right\rangle}

\newcommand{\ZFC}{{\rm ZFC}}
\newcommand{\ZF}{{\rm ZF}}
\newcommand{\ZFCm}{\ZFC^-}
\newcommand{\ZFCmm}{\ZFC\text{\tt -}}

\newcommand{\CH}{{\rm CH}}

\newcommand{\GCH}{{\rm GCH}}

\newcommand{\AC}{{\rm AC}}
\newcommand{\DC}{{\rm DC}}

\newcommand{\cell}[1]{\boxit{\hbox to 17pt{\strut\hfil$#1$\hfil}}}
\newcommand{\head}[2]{\lower2pt\vbox{\hbox{\strut\footnotesize\it\hskip3pt#2}\boxit{\cell#1}}}
\newcommand{\boxit}[1]{\setbox4=\hbox{\kern2pt#1\kern2pt}\hbox{\vrule\vbox{\hrule\kern2pt\box4\kern2pt\hrule}\vrule}}
\newcommand{\Col}[3]{\hbox{\vbox{\baselineskip=0pt\parskip=0pt\cell#1\cell#2\cell#3}}}
\newcommand{\tapenames}{\raise 5pt\vbox to .7in{\hbox to .8in{\it\hfill input: \strut}\vfill\hbox to
.8in{\it\hfill scratch: \strut}\vfill\hbox to .8in{\it\hfill output: \strut}}}
\newcommand{\Head}[4]{\lower2pt\vbox{\hbox to25pt{\strut\footnotesize\it\hfill#4\hfill}\boxit{\Col#1#2#3}}}
\newcommand{\Dots}{\raise 5pt\vbox to .7in{\hbox{\ $\cdots$\strut}\vfill\hbox{\ $\cdots$\strut}\vfill\hbox{\
$\cdots$\strut}}}
\newcommand{\df}{\it} 
\begin{document}

\author{Victoria Gitman}
    \address[V.~Gitman]{Mathematics, The Graduate Center of The City University of New York,
         365 Fifth Avenue, New York, NY 10016}
    \email{vgitman@nylogic.org}
    \urladdr{http://boolesrings.org/victoriagitman}

\author{Joel David Hamkins}
    \address[J.~D.~Hamkins]{Philosophy, New York University, 5 Washington Place New York, New York 10003 \&
         Mathematics, The Graduate Center of The City University of New York,
         365 Fifth Avenue, New York, NY 10016 \&
         Mathematics, College of Staten Island of CUNY, Staten Island, NY 10314}
    \email{jhamkins@gc.cuny.edu}
    \urladdr{http://jdh.hamkins.org}

\author{Thomas A. Johnstone}
    \address[T.~A.~Johnstone]{Mathematics, New York City College of Technology, 300 Jay Street, Brooklyn, NY 11201}
    \email{tjohnstone@citytech.cuny.edu}
    \urladdr{http://websupport1.citytech.cuny.edu/faculty/tjohnstone}

\thanks{The research of each of the authors has been supported
in part by PSC-CUNY research grants from the CUNY Research
Foundation. The second author's research has 
additionally been supported by NSF grant DMS-0800762 and Simons Foundation grant 209252, and the second and third author's research has been supported by CUNY Collaborative Incentive Award 80209-06 20. The third author is grateful to the Kurt G\"odel Research Center at the University of Vienna for the support of his 2009--10 post-doctoral position there, funded in part by grant P20835-N13 from the Austrian Science Fund.}

\begin{abstract}
We show that the theory $\ZFCmm$, consisting of the usual
axioms of \ZFC\ but with the power set axiom
removed---specifically axiomatized by extensionality,
foundation, pairing, union, infinity, separation,
replacement and the assertion that every set can be well-ordered---is weaker than
commonly supposed and is inadequate to establish several
basic facts often desired in its context. For example,
there are models of $\ZFCmm$ in which $\omega_1$ is
singular, in which every set of reals is countable, yet
$\omega_1$ exists, in which there are sets of reals of
every size $\aleph_n$, but none of size $\aleph_\omega$,
and therefore, in which the collection axiom fails; there
are models of $\ZFCmm$ for which the \Los\ theorem fails,
even when the ultrapower is well-founded and the measure
exists inside the model; there are models of $\ZFCmm$ for
which the Gaifman theorem fails, in that there is an
embedding $j:M\to N$ of $\ZFCmm$ models that is
$\Sigma_1$-elementary and cofinal, but not elementary;
there are elementary embeddings $j:M\to N$ of $\ZFCmm$
models whose cofinal restriction $j:M\to \Union j\image M$
is not elementary. Moreover, the collection of formulas that are provably equivalent in $\ZFCmm$ to a $\Sigma_1$-formula or a $\Pi_1$-formula is not closed under bounded quantification.
Nevertheless, these deficits of $\ZFCmm$ are completely repaired by
strengthening it to the theory $\ZFCm$, obtained by using collection
rather than replacement in the axiomatization above. These results
extend prior work of Zarach
\cite{Zarach1996:ReplacmentDoesNotImplyCollection}.
\end{abstract}

\title{What is the theory ZFC without power set?}\maketitle

\section{Introduction}

Set theory without the power set axiom is used in
arguments and constructions throughout the subject and is usually
described simply as having all the axioms of \ZFC\ except for the
power set axiom. This theory arises frequently in the large cardinal
theory of iterated ultrapowers, for example, and perhaps part of its
attraction is an abundance of convenient natural models, including
$\<H_\kappa,{\in}>$ for any uncountable regular cardinal $\kappa$,
where $H_\kappa$ consists of sets with hereditary size less than
$\kappa$. When prompted, many set theorists offer a precise list of
axioms: extensionality, foundation, pairing, union, infinity,
separation, replacement and choice.
For instance, the theory is described as ``\ZFC\
with the Power Set Axiom deleted'' by \cite{kanamori:higher} (ch. 19, p. 244) and as
``set theory without the Power Set Axiom'' by \cite{jech:settheory2003} (ch. 19, p. 354) in their treatment of iterated ultrapowers, with \ZFC\ and ``set theory" referencing the list of axioms given above. Various other authors describe the theory as ``\ZFC\ minus the power-set axiom''  \cite{HandbookOfSetTheory:abraham} (ch. 5, p. 380), ``standard axioms of \ZFC\ excluding the powerset axiom''  \cite{HandbookOfSetTheory:neeman} (ch. 22, p. 1883), or verify replacement (rather than collection) when determining whether a given structure satisfies the theory. The authors of the current paper have in the past described the theory as \ZFC\ without the power set axiom \cite{hamkinsjohnstone:unfoldable, gitman:ramsey}.\footnote{In contrast, Zarach in \cite{zarach:unions_of_zfminus_models} explicitly includes the collection scheme instead of replacement in his definition of $\ZF^-$, as does 
Devlin in \cite{Devlin1984:Constructibility} (p. 119). Jensen, also, reportedly used collection rather than replacement in his course notes. The results of this article
show that using collection in place of replacement is a critical
difference.} Let us denote by $\ZFCmm$ the theory
having the axioms listed above with the axiom of choice taken to mean Zermelo's well-ordering principle, which then implies Zorn's Lemma
as well as the existence of choice-functions. These
alternative formulations of choice are not all equivalent without
the power set axiom as is proved by Zarach in~\cite{zarach:unions_of_zfminus_models},\footnote{Zarach~\cite{zarach:unions_of_zfminus_models} credits  Z. Szczpaniak with first showing that there are models of $\ZF^-$ in which choice-functions exist but Zermelo's well-ordering principle fails.} who initiated the program of establishing unintuitive consistency results for set theory without power set, which we carry on in this paper. Since our aim is to emphasize the weakness of the replacement scheme in the absence of
power set, we opt here for the strongest variant of choice.

In this article, we shall prove that this formulation of set theory
without the power set axiom is weaker than may be supposed and is
inadequate to prove a number of basic facts that are often desired
and applied in its context.
Specifically, we shall prove that the following behavior can occur
with $\ZFCmm$ models.
 \begin{enumerate}
  \item (Zarach) There are models of $\ZFCmm$ in which
      the countable union of countable sets is not
      necessarily countable, indeed, in which
      $\omega_1$ is singular, and hence the
      collection axiom scheme fails.
  \item (Zarach) There are models of $\ZFCmm$ in which
      every set of reals is countable, yet $\omega_1$
      exists.
  \item There are models of $\ZFCmm$ in which for every
      $n<\omega$, there is a set of reals of size
      $\aleph_n$, but there is no set of reals of size
      $\aleph_\omega$.
  \item The \Los\ ultrapower theorem can fail for
      $\ZFCmm$ models.
    \begin{enumerate}
      \item There are models $M\satisfies\ZFCmm$
          with an $M$-normal measure $\mu$ on a
          cardinal $\kappa$ in $M$, for which the
          ultrapower by $\mu$, using functions in $M$, is well-founded,
          but the ultrapower map is not elementary.
      \item Such violations of \Los\ can arise even
          with internal ultrapowers on a measurable
          cardinal $\kappa$, where $P(\kappa)$
          exists in $M$ and $\mu\in M$.
      \item There is $M\satisfies\ZFCmm$ in which
          $P(\omega)$ exists in $M$ and there are
          ultrafilters $\mu$ on $\omega$ in $M$,
          but no such $M$-ultrapower map is
          elementary.
   \end{enumerate}
  \item The Gaifman theorem~\cite{gaifman:ultrapowers} can fail for $\ZFCmm$
      models.
    \begin{enumerate}
      \item There are
          $\Sigma_1$-elementary cofinal maps
          $j:M\to N$ of transitive $\ZFCmm$ models, which are not
          elementary.
      \item There are elementary
          maps $j:M\to N$ of transitive
          $\ZFCmm$ models, such that the canonical cofinal
          restriction $j:M\to \Union j\image M$ is
          not elementary.
    \end{enumerate}
 \item Seed theory arguments can fail for $\ZFCmm$  models. There are elementary embeddings $j:M\to N$ of transitive $\ZFCmm$ models and sets \hbox{$S\of\Union j\image M$} such that the seed hull \hbox{$\X_S=\{j(f)(s)\st s\in [S]^{\ltomega}, f\in M\}$} of $S$ is not an elementary submodel of $N$. In this case, the restriction \hbox{$j:M\to \X_S$} is a cofinal $\Sigma_1$-elementary map that is not elementary.
 \item The collection of formulas that are provably equivalent in $\ZFCmm$  to a $\Sigma_1$-formula or a $\Pi_1$-formula is not closed under bounded quantification.
 \end{enumerate}

The subsequent theorems of this article contain all the details,
including additional undesirable behavior for $\ZFCmm$ models. The
counterexample $\ZFCmm$ models we produce can be arranged to satisfy
various natural strengthenings of $\ZFCmm$, to include such
statements as Hartogs' theorem that $\aleph_\alpha$ exists for every
ordinal $\alpha$, or alternatively, the assertion ``{\it I am
$H_{\theta^\plus}$},'' meaning the assertion that $\theta$ is the
largest cardinal and every set has hereditary size at most $\theta$.

The fact that the replacement axiom scheme does not imply the
collection axiom scheme without the power set axiom was first proved
by Zarach \cite{Zarach1996:ReplacmentDoesNotImplyCollection}, who
knew the situation described in statements (1) and (2), using a
general method similar to the one we employ in this article. This
article should therefore be viewed as an extension of Zarach's results,
particularly to the case of the \Los\ and Gaifman theorems, which we
find interesting because these theorems are extensively used in the
context of set theory without power set.

Although the failure of such properties is often thought to revolve around the axiom of choice---there are after all some famous models of $\ZF+\neg\AC$ in which $\omega_1$ is singular and the \Los\ theorem fails---nevertheless, all our counterexample models satisfy the axiom of choice in any of the usual set formulations, including the existence of choice-functions, Zorn's lemma, and Zermelo's well-ordering principle. Thus, our arguments reveal a separate issue arising from the inequivalence without the power set axiom of the replacement and collection axiom schemes. Meanwhile, there are close connections between our models and several of the usual models of $\ZF+\neg\AC$, such as the Feferman-\Levy\ model, which will be apparent in our constructions. Furthermore, although our models satisfy the axiom of choice, they do not satisfy the class choice scheme, asserting that if every $a\in A$ has some $b$ with $\varphi(a,b,Z)$, then there is a function $f$ with $\forall a{\in}A\, \varphi(a,f(a),Z)$, since this implies the collection axiom. In this way, one might view the collection axiom scheme itself as a kind of choice principle.

The failure of the \Los\ and Gaifman theorems for models of
$\ZFCmm$ may seem troubling---these theorems are fundamental in
the theory of ultrapowers and iterated ultrapowers, where the use of
models of set theory without the power set axiom is
pervasive---but nevertheless all is well with these applications, for
the following reason. Namely, all of the problematic issues
identified in this article for the theory $\ZFCmm$ disappear if one
should simply strengthen it to the theory $\ZFCm$, which is axiomatized by the same list of axioms as in the opening of this article where choice is taken to mean that every set can be well-ordered, but the replacement scheme is replaced by the collection scheme.\footnote{ We believe the choice axiom in $\ZFCm$ should refer to Zermelo's well-ordering principle, since in the context of set theory without power set the existence of choice-functions does not suffice to prove that every set can be well-ordered, even if the collection axiom scheme holds (see~\cite{zarach:unions_of_zfminus_models}).} In particular, the
reader can readily check that $\ZFCm$ suffices to prove that
successor cardinals are regular, and models of $\ZFCm$ satisfy the
\Los\ theorem and the Gaifman theorem and so on. The somewhat higher
minus sign shall serve to remind the reader that the theory $\ZFCm$
is stronger than $\ZFCmm$, and it is this stronger version of the
theory that holds in all applications of \ZFC\ without power set of
which we are aware.  For example, if $\kappa$ is an uncountable
regular cardinal, then $H_\kappa$ is easily seen to satisfy the
collection scheme and hence full $\ZFCm$; also, any model of
$\ZFCmm$ with the global choice axiom, in the
form of a global class well-ordering of the universe, must also satisfy $\ZFCm$, since the
global choice class allows us to transform instances of collection
into instances of replacement, by picking the least witness, and
thereby satisfy them. Thus, in all the uses of models of set theory
without power set of which we are aware, one actually has the
stronger theory $\ZFCm$ anyway, and thereby avoids the problematic
issues of $\ZFCmm$ that we identify in this article.

The main point of this paper, therefore, is to reveal what
can go wrong when one naively uses $\ZFCmm$ in a
set-theoretic argument for which one should really be using
$\ZFCm$, and to point out that if one indeed would use
$\ZFCm$, then all standard arguments carry through as
expected. In other words, our point is that $\ZFCmm$ is the
wrong theory, and in almost all applications, set theorists
should be using $\ZFCm$ instead.

Before continuing, perhaps it would be helpful for us to point out where in the usual arguments that we have mentioned one uses collection rather than merely replacement. For example, when proving that $\omega_1$ is regular, or more generally when proving that the countable union of countable sets is countable, one has a countable sequence $\<X_n\st n<\omega>$ of countable sets $X_n$. One would like to use the axiom of choice to select witnessing bijections $f_n:\omega\cong X_n$ and then, with a pairing
function on $\omega$, use the map $\<n,k>\mapsto f_n(k)$ to map $\omega$ surjectively onto $\Union_n X_n$. The problem with carrying out this argument in $\ZFCmm$ is that in order to apply the axiom of choice in the first place, we would need to have a sequence $\<F_n\st n<\omega>$ of nonempty sets $F_n$ consisting of bijections $f:\omega\cong X_n$ to which to apply it. Without the power set axiom, however, we do not know that the collection of all bijections from $\omega$ to $X_n$ forms a set, and so we somehow need first to reduce to a set of witnesses before applying \AC\ to choose individual witnesses. Thus, we would seem to want the collection scheme, which would exactly allow us to do that, and so the argument does work in $\ZFCm$. A similar issue
arises when proving that successor cardinals $\kappa^\plus$ are regular. The results of \cite{Zarach1996:ReplacmentDoesNotImplyCollection} and also this paper show that with only replacement, one simply cannot push the argument through.

A similar issue arises when proving the \Los\ theorem in
the forward direction of the existential case. One has the
ultrapower $j:M\to \Ult(M,\mu)$ and $M\satisfies\exists x\,
\varphi(x,f(\alpha))$ for $\mu$-almost every $\alpha$, where the ultrapower is constructed using functions in $M$. What
one would like to do is to apply the axiom of choice in
order to select a witness $x_\alpha$ such that
$M\satisfies\varphi(x_\alpha,f(\alpha))$, for each $\alpha$
for which there is such a witness. But in order to apply
the axiom of choice here, one must first know that there is
a set of witnesses $x$ such that $\varphi(x,f(\alpha))$
from which to choose, which may be accomplished using collection.  But without
collection, it follows by our results that the argument
simply cannot succeed. Note that replacement in $M$ suffices to prove the \Los\ theorem for $\Delta_0$-formulas since if there is a function $g\in M$ such that $M\satisfies \exists x {\in} g(\alpha)\,\varphi(x,f(\alpha))$ for $\mu$-almost every $\alpha$, then the witnesses can be chosen out of $\;\Union \ran (g)$, which is a set by replacement. Since ultrapower embeddings of models of $\ZFCmm$ are cofinal\footnote{An embedding $j:M\to N$ is \emph{cofinal} if every $x\in N$ is an element of $j(y)$ for some $y\in M$.} by replacement, it follows that they are always $\Sigma_1$-elementary. We will produce a model of $\ZFCmm$ with an ultrapower embedding that is not $\Sigma_2$-elementary and for which the \Los\ theorem fails already for $\Sigma_1$-formulas. This will also demonstrate a failure of the Gaifman theorem.

In its full generality, the Gaifman theorem~\cite{gaifman:ultrapowers} states
that if $M\satisfies \ZFCm$ and $j:M\to N$ is a $\Delta_0$-elementary cofinal map, then $j$ is fully elementary.\footnote{In Gaifman's original formulation, the theorem stated that if $M$ is a model of Zermelo's set theory ($\ZF$ axioms with the separation scheme but no replacement scheme),  and $j:M\to N$ is a $\Delta_0$-elementary cofinal map, then it is fully elementary. Gaifman, then pointed out that the existence of cartesian products in $M$ suffices in the place of power sets, and thus, in our context, the theorem holds if $M\satisfies \ZFCm$.} No additional assumptions are made on $N$, and the models $M$ and $N$ need not be transitive. It is easy to see that $N$ satisfies the pairing axiom, since if $x,y$ are any elements in $N$, then by cofinality there are $a,b\in M$ such that $x\in j(a)$ and $y\in j(b)$, and so, by replacement in $M$, there is the cartesian product $C=a\times b$ so that \hbox{$M\satisfies \forall u{\in} a\,\forall v{\in} b \, \exists w{\in} C \, w{=}\<u,v>$}, which is a $\Delta_0$-formula. It follows that, in both $M$ and $N$, blocks of like quantifiers in formulas may be contracted to a single quantifier through applications of the pairing axiom. The proof now proceeds by induction on the complexity of formulas, and since $j$ is $\Sigma_1$-elementary by cofinality,  the critical case of the induction occurs when $j$ is assumed to be $\Sigma_n$-elementary for some $n>0$ and  \hbox{$M\satisfies\forall x\in a\,\exists y\, \varphi(x,y,p)$} for some formula $\varphi\in \Pi_{n-1}$ and some sets $a,p\in M$. Using collection in $M$, there is a set $b$ such that \hbox{$M\satisfies \forall x{\in} a\,\exists y{\in}b\,\varphi(x,y,p)$}, and since $a\times b$ exists in $M$, one may use separation in $M$ to obtain a set $C$ such that $M$ satisfies \hbox{$\forall x{\in} a\exists y{\in} b\,\langle x,y\rangle \in C$} and \hbox{$\forall x{\in} a\forall y{\in} b\,\big[\langle x,y\rangle\in C\Rightarrow \varphi(x,y,p)\big]$}.
The first statement is $\Sigma_0$, and the second is $\Pi_{n-1}$, so that both statements transfer to $N$ by the inductive hypothesis. It follows that $N\satisfies \forall x{\in} j(a)\,\exists y{\in} j(b)\,\varphi(x,y,j(p))$, completing the proof of the critical case.\footnote{The proof of Gaifman's result becomes a much easier induction on the complexity of formulas, if one assumes at the outset that $N$ is a model of $\ZFCm$, and this folklore version of the theorem was already known before Gaifman's result according to \cite{gaifman:ultrapowers}.} We shall show that the use of collection in $M$ is essential to this argument by producing a cofinal $\Delta_0$-elementary map \hbox{$j:M\to N$} of $\ZFCmm$ models that is not elementary.

An analogous issue as in the \Los\ theorem arises when proving that for a given elementary embedding $j:M\to N$ and some $S\of \Union j\image M$ the \emph{seed hull} of $S$ via $j$, meaning the structure $\X_S=\{j(f)(s)\st s\in [S]^{\ltomega}, f\in M\}$, is an elementary submodel of $N$. This is usually shown by  verifying the Tarski-Vaught test for $\X_S\of N$, and so one has that $N\satisfies \exists y \,\varphi(y,j(f)(s))$ where $s\in j(D)$ for some $D\in M$. One would like to use a Skolem function $g\in M$ for $\varphi(y,f(x))$, meaning that for each $x\in D$, one lets $g(x)$ be any witness $y$ such that $\varphi(y,f(x))$ holds in $M$, if such a witness exists. But, the usual proof that such $g$ exists in $M$ uses collection in $M$, since one can appeal to \AC\ in $M$  to choose witnesses only after sufficiently many such witnesses have been collected to a set in $M$. This shows that $\X_S\elesub N$ in the case when $M\satisfies\ZFCm$, but in our context, when $M\satisfies\ZFCmm$, we will find elementary embeddings $j:M\to N$ and sets $S$ of seeds such that $\X_S\not \elesub N$. Note that replacement in $M$ suffices to see that $\X_S$ is $\Delta_0$-elementary in $N$, since in this case the existential quantifier in the Tarski-Vaught test may be bounded by some $j(h)(s)$ and therefore the witnesses for the Skolem function $g\in M$ can be chosen out of $\Union \ran (h)$. The map $j:M\to \X_S$ is thus cofinal and hence $\Sigma_1$-elementary, but not fully elementary, and so is the map $\pi\circ j:M\to \ran(\pi)$, where $\pi$ denotes the Mostowski collapse of $\X_S$, which means that both maps demonstrate a failure of the Gaifman theorem. Moreover, our maps will allow for seed hulls $\X_S$ that are not generated by a single seed, so that we will get these failures for embeddings that are not ultrapower maps. 

\section{Badly behaved models of $\ZFCmm$}

Let us now begin our project by describing a variety of
models of $\ZFCmm$, in each case making observations about
how that model reveals a compatibility of various
undesirable behaviors with the theory $\ZFCmm$. Although
there is some redundancy in this bad behavior of the models
in that several different example models reveal the same
deficiency in $\ZFCmm$, we have nevertheless included all
the examples to illustrate the range and flexibility of the
constructions. We shall organize our presentation
principally around the construction methods, rather than
around the various undesirable behavior for models of
$\ZFCmm$.

\subsection{The \Levy\ collapse of
$\aleph_\omega$}\label{Section.LevyCollapseAleph_omega} Let
$\Coll(\omega,{\lt}\aleph_\omega)$ be the \Levy\ collapse of
$\aleph_\omega$, meaning the finite-support product
$\prod_n\Coll(\omega,\aleph_n)$, and suppose that
$G\of\Coll(\omega,{\lt}\aleph_\omega)$ is $V$-generic. The cardinal
$\aleph_\omega$ is collapsed to $\omega$ in $V[G]$, but it
remains a cardinal in every model $V[G_n]$, where $G_n=G\cap \P_n$
is the restriction of $G$ to the collapse forcing $\P_n=\prod_{k\leq
n}\Coll(\omega,\aleph_k)$ which proceeds only up to $\aleph_n$. Our
desired model is
$$W=\Union_{n<\omega} V[G_n],$$ which is the union of the forcing
extensions $V[G_n]$ arising during initial segments of the collapse. The model $W$ is a definable class in $V[G]$ using a parameter from $V$ determined by $\Coll(\omega,{\lt}\aleph_\omega)$ to define the ground model $V$ (by a result of Laver's \cite{Laver2007:CertainVeryLargeCardinalsNotCreated}) and the filter $G$. Although the definability of $W$ will not be required for this argument, it will be used in the later constructions. Note that $W$ depends not only on $V[G]$, but also on the way
that $G$ is presented, and there are automorphisms of the forcing
that do not preserve $W$.\footnote{For example, one could view each
stage of the forcing as also adding a Cohen generic real via an
isomorphism of $\P_n$ with the product of $\P_n$ and the Cohen
poset, and then consider an automorphism in which the sequence of
these reals is added before the first collapsing stage via an
isomorphism of $\Coll(\omega,{<}\aleph_\omega)$ with the countable
product of Cohen posets and $\Coll(\omega,{<}\aleph_\omega)$, so that
for one generic filter, the sequence is not in $W$, but for the
automorphic copy of the generic filter, it is.} Nevertheless, any
automorphism of any particular $\P_n$ induces an automorphism of
$\P$ that preserves $W$. Note that every stage $m$ of the forcing
has for any two conditions with the same domain, an automorphism
generated by mapping one to the other, and combining these
automorphisms for $m\leq n$ produces automorphisms of $\P_n$
that will be used later in the argument.

Let us now verify that $W\satisfies\ZFCmm$, by checking each axiom
in turn. Most of them follow easily. For example, extensionality and
foundation hold in $W$, because it uses the $\in$-relation of $V[G]$
and is transitive; union and pairing hold since $W$ is the union of
a chain of models of \ZFC, which each have the required union and
pairing sets; infinity holds since $\omega\in V\of W$; and Zermelo's
well-ordering principle holds in $W$, because every set in $W$ is in
some $V[G_n]$, where it has a well-order that survives into $W$; and
separation follows from replacement, and so needn't be considered
separately.

So it remains to verify only the replacement axiom scheme in $W$.
Suppose that $A$ and $z$ are sets in $W$, and that
$W\satisfies\forall a{\in}A\,\exists! b\, \varphi(a,b,z)$, where the
exclamation point expresses that there is a unique such $b$. We may
assume without loss of generality that $A,z\in V$, since otherwise
we have $A,z\in V[G_n]$ for some $n$ and we may replace $V$ with
$\Vbar=V[G_n]$, using the fact that forcing with the tail of the
product gives rise to the same $W$. We claim next that for any $a\in
A$, the unique witness $b$ for $a$ is also in $V$. Fix any $a\in A$,
and let $b$ be the unique witness such that
$W\satisfies\varphi(a,b,z)$. It must be that $b\in V[G_n]$ for some
$n$, and so $b=\dot b_G$ for some $\P_n$-name $\dot b$. Note that we
specifically chose a $\P_n$-name to ensure that its interpretation
by any generic $\Gbar$ will end up in $V[\Gbar_n]$, a circumstance
crucial to the later part of the argument. Suppose that some
condition $p\in G$ forces $\varphi^{\dot W}(\check a,\dot b,\check
z)$, where $\dot W$ is the forcing language definition of the class
$W$ using the canonical name $\dot G$ for the generic
filter. Note that the stage $n+1$ forcing
$\Coll(\omega,\aleph_{n+1})$ is isomorphic to the product
$\P_n\times \Coll(\omega,\aleph_{n+1})$, and so we may view the
stage $n+1$ forcing as first adding another mutually generic filter
$\Gbar_n\of\P_n$ and then performing the rest of the forcing. That
is, $\P$ is isomorphic to the forcing that forces with two copies of
$\P_n$, first using the actual copy and then using an additional
copy, before the product proceeds with stage $n+1$ and the
rest. By swapping these two copies of $\P_n$ inside
$\P_{n+1}$, we produce a slightly different $V$-generic filter
$\Gbar$, such that $G_n$ and $\Gbar_n$ are mutually generic for
$\P_n$, but $V[G]=V[\Gbar]$. Since we are rearranging the filter
only within $\P_{n+1}$, it follows that $\dot W_G=\dot W_{\Gbar}$.
We may furthermore assume without loss of generality that
$p\in\Gbar$, by applying if necessary an additional automorphism (as
described above) to the second copy of $\P_n$ that we constructed in
$\P_{n+1}$, before we perform the swap of the two copies of $\P_n$.
It follows that $\varphi^W(a,\dot b_{\Gbar},z)$, and so by the
uniqueness of $b$, we have $b=\dot b_G=\dot b_\Gbar$.
This equation implies that $b$ is in $V[G_n]\intersect V[\Gbar_n]$,
which is equal to $V$ since these filters are mutually generic, by a result of Solovay \cite{Solovay1970:AModelInWhichEverySetIsLebesgueMeasurable}.
Thus, $b\in V$ as we claimed. Now, for any given $a\in A$ the
question of whether a given $b\in V$ has $\varphi^W(a,b,z)$ must be
decided by $\one$, since $\P$ is weakly homogeneous by automorphisms
(as described above) affecting only finitely many coordinates and
such automorphisms do not affect the value of $\dot W$.\footnote{A
forcing $\P$ is \emph{weakly homogeneous} if for any two conditions
$p$ and $q$ there is an automorphism $\pi$ of $\P$ such that
$\pi(p)$ is compatible with $q$. Since
$\pi(\boolval{\varphi(\tau)})=\boolval{\varphi(\tau^\pi)}$, where
$\tau^\pi$ is the name produced by hereditary application of $\pi$,
it follows for such forcing that the Boolean value of any statement
whose parameters are not affected by $\pi$ is either $0$ or $\one$.}
Thus, the set $\set{b\st \exists a{\in}A\, \varphi^W(a,b,z)}$ exists
in $V$ by replacement in $V$, and hence also exists in $W$ as
desired. This completes the argument that $W$ is a model of replacement, and so $W\satisfies\ZFCmm$.

We shall now make further observations about this model $W$ in order
to prove that $\ZFCmm$ is consistent with various undesirable
behaviors. The first few observations were made already by Zarach
\cite{Zarach1996:ReplacmentDoesNotImplyCollection}, using a method
fundamentally similar to ours. In the subsequent sections, we shall
modify this construction in order to obtain failures of the \Los\
and Gaifman theorems, among others.

\begin{theorem}[Zarach]\label{Theorem.Omega1Singular}
It is consistent with $\ZFCmm$ that $\omega_1$
exists but is singular and hence that a countable union of
countable sets can be uncountable.
\end{theorem}

\begin{proof}Consider the model $W$ as constructed in the \Levy\
collapse of $\aleph_\omega$ above. Note that every cardinal
$\aleph_n^V$ is collapsed in $V[G_n]$ and hence in $W$, but
$\aleph_\omega^V$ remains a cardinal in every $V[G_n]$ and hence
also in $W$. Thus, $\omega_1^W=\aleph_\omega^V$, which has
cofinality $\omega$ in $V$ and hence in $W$, as witnessed by the
sequence $\<\aleph_n^V\st n\lt\omega>\in V\of W$. So $W$ satisfies
that $\omega_1$ is singular. In particular, $W$ satisfies that
$\omega_1$ is a countable union of countable sets, as
$\omega_1^W=\Union\set{\aleph_n^V\st n\in \omega}$.
\end{proof}

\begin{theorem}[Zarach]\label{Theorem.Collection}
It is consistent with $\ZFCmm$ that the
collection scheme fails. Hence, replacement and collection
are not equivalent without the power set axiom.
\end{theorem}

\begin{proof}The collection scheme
fails in the model $W$ above because for each countable
ordinal $\alpha$ in $W$ there is a surjective function
$f:\omega\surj\alpha$, but there is no set $B$ in $W$
collecting a family of such functions, since every $B$ in
$W$ arises in some $V[G_n]$ and therefore contains no
functions collapsing the ordinals $\alpha$ in the interval
$[\aleph_{n+1}^V,\aleph_\omega^V)$.
\end{proof}
In the proof above, we could have also obtained a violation of collection using domain $\omega$ instead of $\omega_1$ by considering the sequence $\< \aleph_n^V\mid n<\omega>$, which is an element of $V$ and hence of $W$, and observing that $W$ has bijections $f:\omega\surj\aleph_n$, but a family of such functions cannot be collected.

\begin{theorem}[Zarach]\label{Theorem.SetsOfRealsCountableYetomega1}
It is consistent with $\ZFCmm$ that every set of
reals is countable, yet $\omega_1$ exists.
\end{theorem}

\begin{proof}
Consider the model $W$ constructed as above but starting with a ground model $V$ in which $2^{{\lt}\aleph_\omega}=\aleph_\omega$.
Every set of
reals in such a $W$ is in $V[G_n]$ for some $n$, and the reals of this
model have size $\aleph_m$ for some $m$ by a nice name counting argument and the assumption $2^{
{\lt}\aleph_\omega}=\aleph_\omega$. Thus, the reals of every $V[G_n]$ become countable in some further $V[G_m]$ and hence in $W$, so
every set of reals in $W$ is countable in $W$. But $W$ satisfies
that $\omega_1$ exists and indeed $\aleph_\alpha$ exists for every
ordinal $\alpha$, since above $\aleph_\omega$ the cardinals of $W$
agree with the cardinals of $V$.
\end{proof}

The construction is general, and it adapts to forcing with the \Levy\ collapse of $\aleph_\kappa$ to $\kappa$, meaning the bounded-support product $\Coll(\kappa,{\lt}\aleph_\kappa)=\prod_{\beta<\kappa}\Coll(\kappa,\aleph_\beta)$.  If $G_\gamma=G\intersect \P_\gamma$, where $\P_\gamma=\prod_{\beta\leq\gamma}\Coll(\kappa, \aleph_\beta)$ is the forcing up to $\gamma$, then let $W=\Union_{\gamma<\kappa}V[G_\gamma]$. As long as $\kappa$ is a regular cardinal with $\aleph_\beta^{\ltkappa}<\aleph_\kappa$ for all $\beta<\kappa$, it follows by essentially the same arguments as before that $W$ satisfies $\ZFCmm$ but not collection, and that the cardinal $\kappa^+$ exists in $W$ but has cofinality $\kappa$ there.  Moreover, the model $W$ is closed in $V[G]$ under sequences of length less than $\kappa$, so this construction provides a general method for obtaining badly behaved inner models $W\satisfies\ZFCmm$ that are as closed as desired.

As we did in this section, we will start many arguments in this article with a model $V$ of
\ZFC\ and produce a model of the desired theory with
$\ZFCmm$ by finding an inner model of a forcing extension
of $V$. By the usual forcing methods, such as taking a
quotient of a Boolean-valued model, these arguments show
that if \ZFC\ is consistent, then so is $\ZFCmm$ with the
stated extra properties. One can omit the need for this
consistency assumption and prove in \ZFC\ alone that there
are transitive models of $\ZFCmm$ with the desired
properties, simply by forcing over a suitably large
$H_{\theta^\plus}$ rather than $V$ itself. For example, one
could first make a countable elementary submodel of such an
$H_{\theta^\plus}$, and then build the generic extension by
meeting the countably many dense sets.
We carry out several such arguments in section \S\ref{Section.NoLargeCardinals}.
In any case we find the models $W$ obtained by forcing
over models of \ZFC\ to be more striking, since they have $V$ as an
inner model, and for example also satisfy Hartogs' theorem that for
every ordinal $\alpha$, the cardinal $\aleph_\alpha$ exists, and
many other attractive properties.

\subsection{Pumping up the continuum}\label{Section.PumpingUpTheContinuum}
In the next example, we shall force to pump up the continuum instead
of forcing to collapse cardinals as in the previous example. Let us
start in a model $V$ of $\ZFC$ in which $2^\omega\lt\aleph_\omega$,
and let $G\of\Add(\omega,\aleph_\omega)$ be $V$-generic for the
forcing to add $\aleph_\omega$-many Cohen reals.  Let
$W=\Union_{n<\omega}V[G_n]$, where
$G_n=G\intersect\Add(\omega,\aleph_n)$ is the initial segment of the
forcing adding only $\aleph_n$-many Cohen reals. We argue that
$W\satisfies\ZFCmm$. Once again, most of the axioms are easy to
verify, since $W$ is the union of an increasing chain of transitive \ZFC\ models
$V[G_n]$. It is non-trivial to verify only replacement, which we
handle by a similar argument as above. Namely, suppose that $A,z\in
W$ and $W\satisfies\forall a{\in}A\,\exists!b\,\varphi(a,b,z)$. As
before, we assume that $A,z\in V$, since otherwise they are in some
$V[G_k]$, which we may regard as the new ground model. For each
$a\in A$, we claim again that the witness $b$ for which
$\varphi^W(a,b,z)$ must be in $V$. In any case $b\in V[G_n]$ for
some $n$, and so there is some $\P_n$-name $\dot b$ with $b=\dot
b_G$ and a condition $p\in G$ with $p\forces\varphi^{\dot W}(\check
a,\dot b,\check z)$. We may view the forcing to add $G$ as
consisting of first adding $\aleph_n$-many Cohen reals, and then
adding $\aleph_n$-many more Cohen reals, and then adding the rest of
them, since this is an isomorphic presentation of the forcing. By
using the automorphism that swaps these two mutually generic blocks
of $\aleph_n$-many Cohen reals, we may rearrange the filter $G$ to
construct another filter $\Gbar$ for which $V[G]=V[\Gbar]$, but
$V[G_n]$ and $V[\Gbar_n]$ are mutually generic extensions by
$\Add(\omega,\aleph_n)$, while still having $\dot W_G=\dot
W_{\Gbar}$. As before, we may also apply an additional automorphism
to the second copy of $\Add(\omega,\aleph_n)$ if necessary when
constructing $\Gbar$ and assume without loss of generality that
$p\in\Gbar$. It follows that $W\satisfies\varphi(a,\dot
b_{\Gbar},z)$, and consequently, by the uniqueness of $b$, that
$b=\dot b_G=\dot b_\Gbar$. Thus, $b$ lies in both $V[G_n]$ and
$V[\Gbar_n]$, but the intersection of these models is $V$ by mutual
genericity, and so $b\in V$, as we claimed. To complete the
argument, observe that since $\P$ is weakly homogeneous by
automorphisms not affecting $\dot W$, we have that for any given
$a\in A$ the question of whether a given $b\in V$ has
$\varphi^W(a,b,z)$ must be decided by $\one$. So we may in $V$ build
the set $\set{b\st \exists a{\in}A\, \varphi^W(a,b,z)}$ by using
replacement in $V$. So this set exists in $W$, and we have verified
the replacement scheme in $W$.

\begin{theorem}\label{Theorem.SetsOfRealsAleph_n}
It is consistent with $\ZFCmm$ that
$\aleph_\omega$ exists and for each $n$ there are sets of
reals of size $\aleph_n$, but no set of reals of size
$\aleph_\omega$. In particular, there is no set of reals of largest
cardinality, and this violates the collection
scheme.
\end{theorem}

\begin{proof}
Consider the model $W$ we just constructed and recall
that we assumed $2^\omega<\aleph_\omega$ in the ground model $V$.
Since the forcing $\Add(\omega,\aleph_\omega)$ has the countable
chain condition, it preserves all cardinals to $V[G]$ and hence also
to $W$.  The reals of
$V[G_n]$ exist in $W$ and have size at least $\aleph_n$ there; hence
$W$ has sets of reals of size exactly $\aleph_n$ by Zermelo's
well-ordering principle. But every set of reals in $W$ is in
$V[G_m]$ for some $m<\omega$ and hence must have size less than
$\aleph_\omega$.
\end{proof}

As in the previous section, the construction here is also quite general, and it adapts easily in order to produce models $W$ of $\ZFCmm$ that violate collection but are highly closed in the overall forcing extension $V[G]$. For example, if $\kappa$ is any regular cardinal with $2^\omega < \aleph_\kappa$  and $G\of \Add(\omega,\aleph_\kappa)$ is $V$-generic for the forcing to add $\aleph_\kappa$-many Cohen reals, then let $W=\Union_{\gamma < \kappa} V[G_\gamma]$, where $G_\gamma=G\intersect \Add(\omega, \aleph_\gamma)$. It follows essentially by the same arguments as before that $W$ satisfies $\ZFCmm$ but not collection, and that $W$ has sets of reals of size $\aleph_\gamma$ for each $\gamma<\kappa$, but no set of reals of size $\aleph_\kappa$.
Moreover, the model $W$ is closed under sequences of length less than $\kappa$ in $V[G]$. Other generalizations are possible also, and we could have instead added subsets to $\omega_1$ or $\omega_2$ or to other regular cardinals, and made similar conclusions.

Let us now consider a natural strengthening of both collection and replacement. The {\df reflection principle} scheme is the assertion of any formula $\varphi$, that for any given set $x$ there is a transitive set $A$ containing $x$ as a subset such that $\varphi$ is absolute between $A$ and the universe. The reflection principle scheme is provable in \ZFC\ by appealing to the von Neumann $V_\alpha$ hierarchy, which does not exist when the power set axiom fails. The reflection principle scheme implies the collection axiom scheme directly, and in the presence of separation, it also implies the replacement axiom scheme. Failure of collection therefore implies the failure of the reflection principle scheme, as in the previously constructed models $W$ of $\ZFCmm$.  It is an interesting open question whether the reflection principle scheme is provable in $\ZFCm$. We suspect, just as Zarach does in
\cite{Zarach1996:ReplacmentDoesNotImplyCollection}, that it is not.

The reflection principle scheme is, over $\ZFCm$, equivalent to a version of the
axiom of choice that we call the \emph{dependent choice scheme}, which is the natural class version of Tarski's principle of dependent choices for
definable relations. Specifically, the $\DC$-scheme asserts of any
formula $\varphi$, that for any parameter $z$, if for every $x$
there is $y$ with $\varphi(x,y,z)$, then there is an
$\omega$-sequence $\<x_n\st n<\omega>$ such that $\forall
n\,\varphi(x_n,x_{n+1},z)$. In other words, if $\varphi$ defines a
relation having no terminal nodes, then we can make $\omega$-many dependent choices to find an $\omega$-path through
this relation.  The reflection principle scheme implies the $\DC$-scheme, since it reduces instances of the $\DC$-scheme to set instances of $\DC$, which then follow from \AC. Conversely,
to obtain reflection for a particular formula $\varphi$ from the \DC-scheme,  we first use collection to see that any given set $x$ can be extended to a transitive set containing existential witnesses for all subformulas of $\varphi$ with parameters from $x$, and then we use the \DC-scheme to chose an $\omega$-path of such extensions.
The next theorem shows that the use of collection when proving this converse direction is essential.

\begin{theorem}[Zarach]\label{theorem.DCDoesNotImplyReflection}
It is consistent with $\ZFCmm$ that the $\DC$-scheme holds, but the reflection principle scheme fails.
 \end{theorem}

\begin{proof} Suppose that $2^\omega<\aleph_{\omega_1}$ and that $G\of\Add(\omega,\aleph_{\omega_1})$ is the forcing to add $\aleph_{\omega_1}$-many Cohen reals. If $W=\Union_{\gamma<\omega_1}V[G_\gamma]$ where $G_\gamma=G\intersect \Add(\omega,\aleph_\gamma)$, then as we discussed earlier in this section, $W$ satisfies $\ZFCmm$ but not collection and is closed under countable sequences in $V[G]$. The
$\DC$-scheme holds in $W$, since it holds in $V[G]$ and $W$ is a definable class with $W^\omega\of W$ in $V[G]$. However, the reflection principle scheme fails in~$W$, since collection does.
\end{proof}
For any infinite cardinal $\kappa$, \Levy\ \cite{levy:DCkappa} introduced the principle $\DC_\kappa$, which is the assertion that for any set $A$ and any binary set relation $R$, if for each sequence $\vec s\in A^{\ltkappa}$ there is a $y\in A$ such that $\vec s$ is $R$-related to $y$, then there is a $\kappa$-sequence $\<x_\xi \st \xi<\kappa>$ such that for each $\alpha<\kappa$ the initial sequence $\<x_\xi \st \xi<\alpha>$ is $R$-related to  $x_\alpha$. This natural generalization of Tarski's principle of dependent choices allows for $\kappa$-many dependent choices, rather than just $\omega$-many, and it is easy to see that $\DC_\omega$ is equivalent to the usual principle of dependent choices. Jech showed in \cite{jech:DCkappa} that it is relatively consistent with $\ZF$ that $\DC_\alpha$ holds for all $\alpha$ below any given regular $\kappa$ but $\DC_\kappa$ fails.

In our context, namely in the theory $\ZFCmm$ where every set can be well-ordered,  we consider the natural class version of $\DC_\kappa$, namely the principle that we call the \emph{$\DC_\kappa$-scheme}, which asserts of any formula $\varphi$ that for any parameter $z$, if for every $x$ there is $y$ with $\varphi(x,y,z)$, then there is a function $f$ with domain $\kappa$ such that $\forall \xi{<}\kappa\,\varphi(f{\restrict} \xi,f(\xi),z)$.
 It is easy to see that the $\DC_\omega$-scheme is equivalent over $\ZFCmm$ to the $\DC$-scheme. The reflection principle scheme implies the $\DC_\kappa$-scheme for all cardinals $\kappa$, since it reduces it to set instances of $\DC_\kappa$, which then follow from \AC . But, as the next theorem shows, under $\ZFCmm$ the $\DC_\kappa$-scheme is not strong enough to prove the reflection principle scheme, or $\DC_{\lambda}$ if $\lambda>\kappa$.

\begin{theorem}\label{theorem.DC_kappaDoesNotImplyReflection}
Suppose that $\kappa$ is any regular cardinal with $2^\omega<\aleph_\kappa$ and that $G\of\Add(\omega,\aleph_\kappa)$ is $V$-generic. If $W=\Union_{\gamma<\kappa}V[G_\gamma]$ where $G_\gamma=G\intersect \Add(\omega,\aleph_\gamma)$, then $W\satisfies\ZFCmm$ has the same cardinals as $V$ and the $\DC_\alpha$-scheme holds in $W$ for all $\alpha<\kappa$, but the $\DC_\kappa$-scheme and the reflection principle scheme fail. In particular, it is consistent with $\ZFCmm$ that the $\DC$-scheme fails.
 \end{theorem}

\begin{proof} As we discussed earlier in this section, $W$ is a model of $\ZFCmm$ but not collection, it has sets of reals of size $\aleph_\gamma$ for each $\gamma<\kappa$, but no set of reals of size $\aleph_\kappa$, and $W^{\ltkappa}\of W$. Using the chain condition, $W$ and $V$ have the same cardinals. Thus, for each $\alpha<\kappa$, the $\DC_\alpha$-scheme holds in $W$, but the reflection principle scheme fails. To see that $W$ does not satisfy the $\DC_\kappa$-scheme, let
$\varphi(x,y)$ assert that $y$ is an infinite set of reals,
and if $x$ is a sequence of sets of reals, then $|\Union x|<|y|$. Because
there are increasingly large sets of reals in $W$, it
follows that for  each sequence $\vec s \in W^{\ltkappa}$ of sets of reals there is a $y\in W$ such that $\varphi(\vec s,y)$ holds in $W$. But there
is no $\kappa$-path through this relation in $W$, because
the union of any such sequence would be a set of reals in
$W$ of size at least $\aleph_\kappa$, a contradiction. Lastly note that if $\kappa=\omega$, the \DC-scheme fails.
\end{proof}

Zarach showed already in \cite{Zarach1996:ReplacmentDoesNotImplyCollection} that there is a model of $\ZFCmm$ in which the \DC-scheme fails. He also showed that over $\ZFCmm$, the scheme $\forall \alpha\;\DC_\alpha$ implies the collection scheme and therefore, if collection fails then there must be a cardinal $\kappa$ such that $\DC_\kappa$ fails as well.

\subsection{The \Levy\ collapse of an inaccessible
cardinal}\label{Section.LevyCollapseInaccessible} Suppose now that
$\kappa$ is an inaccessible cardinal and that $G$ is $V$-generic for
the \Levy\ collapse of $\kappa$, meaning the finite-support product
$\Coll(\omega,\ltkappa)=\prod_{\beta\lt\kappa}\Coll(\omega,\aleph_\beta)$.
Let $W=\Union_{\gamma<\kappa}V[G_\gamma]$, where
$G_\gamma=G\intersect\P_\gamma$, where
$\P_\gamma=\prod_{\beta\leq\gamma}\Coll(\omega,\aleph_\beta)$ is the forcing
up to $\aleph_\gamma$. Note that $W^\omega\of W$ in $V[G]$. The argument that $W\satisfies\ZFCmm$ is identical to those previously given and relies on the fact that $\Coll(\omega,\aleph_{\gamma+1})$ may be viewed via forcing equivalence as first performing another copy of $\P_\gamma=\prod_{\beta\leq\gamma}\Coll(\omega,\aleph_\beta)$, and then performing the rest of the collapse.

The model $W$ possesses another interesting feature. Note that the filter $G_\gamma$ is
coded by a real in $V[G]$, and every real of
$V[G]$ appears in some $V[G_\gamma]$. So an equivalent
description of $W$ in $V[G]$ is as the union
$W=\Union_{r\in\R}V[r]$ (where $V[r]$ denotes the closure of $V$ and $r$ under the \Godel\ operations).
\begin{theorem}\label{Theorem.EverySetRealsCountableOmega1Regular}
Relative to an inaccessible cardinal, it is consistent with $\ZFCmm$ that $\omega_1$
exists and is regular, but every set of reals is countable, which implies that the collection scheme fails.
\end{theorem}

\begin{proof}
Consider the model $W$ as constructed just
previously. The cardinal $\kappa$ becomes $\omega_1$ in
$V[G]$ and hence also in $W$, and remains regular there.
Every set of reals in $W$ is in $V[G_\gamma]$ for some
$\gamma\lt\kappa$, and becomes countable at a later stage
and hence countable in $W$. This implies that the
collection scheme fails in $W$, because for each countable
ordinal $\alpha$, there is a function
$f:\omega\cong\alpha$, but there is no set containing such
functions for every $\alpha$, since from such a set we
could construct in $W$ an uncountable set of reals.
\end{proof}

To see the failure of collection in the proof of theorem~\ref{Theorem.EverySetRealsCountableOmega1Regular}, it was crucial that every set of reals was countable. Indeed, one of Zarach's intriguing results on set theory without power set shows that it is relatively consistent with $\ZFCm$ that Hartogs' theorem holds, but there are unboundedly many cardinals whose powersets are proper classes with all subsets of a certain bounded size. For example, he provides in~\cite{zarach:unions_of_zfminus_models} a model of $\ZFCm$ for which $\aleph_\alpha$ exists for each ordinal $\alpha$, where $P(\omega)$ is a proper class, but every set of reals has size at most $\omega_1$.

Note also that if $W\satisfies\ZFCmm$ is a model of Hartogs' theorem
 in which $\omega_1$ exists and is regular and every set of reals is countable, then $\omega_1$ is inaccessible in $L^W\models \ZFC$, and indeed, is inaccessible to reals, for
otherwise we would find an uncountable set of reals in some
$L[x]^W$, which would remain uncountable in $W$. So the use of the
inaccessible cardinal is necessary for any construction that obtains
$W$ as above.

\subsection{The \Levy\ collapse of a measurable
cardinal}\label{Section.LevyCollapseMeasurable} Let us now
turn to a version of the construction providing a violation
of the \Los\ theorem. Namely, in theorem
\ref{Theorem.LevyCollapseMeasurable} we show that in the
\Levy\ collapse $V[G]$ of a measurable cardinal $\kappa$,
the inner model $W$ as constructed in
section \S\ref{Section.LevyCollapseInaccessible} has a
definable ultrafilter $\mu^*$ on $\kappa=\omega_1^{V[G]}$,
whose ultrapower is well-founded, but the ultrapower map is
not elementary. Thus, it is relatively consistent with
$\ZFCmm$ that this version of the \Los\ theorem fails for
ultrapowers, where the ultrafilter is fully amenable to the
model and indeed definable over the model and the
ultrapower is well-founded. Furthermore, we show that the
Gaifman theorem fails for this ultrapower embedding, since
it is $\Sigma_1$-elementary and cofinal, but not fully
elementary.

\begin{theorem}\label{Theorem.LevyCollapseMeasurable}
If\/ $V[G]$ is the forcing extension by the \Levy\ collapse
$G\of\Coll(\omega,\ltkappa)$ of a measurable cardinal
$\kappa$ and $W=\Union_{\gamma<\kappa}V[G_\gamma]$, where
$G_\gamma=G\intersect\P_\gamma$ and
\hbox{$\P_\gamma=\prod_{\beta\leq\gamma}\Coll(\omega,\aleph_\beta)$}, then:
 \begin{enumerate}
  \item $W\satisfies\ZFCmm$.
  \item $W^{\ltkappa}\of W$ in $V[G]$.
   \item In $W$ there is a definable $W$-normal measure
      $\mu^*$ on $\kappa$.
   \item The ultrapower $\Mbar\cong\Ult(W,\mu^*)$ taken in
      $V[G]$ using functions on $\kappa$ in $W$ is
      well-founded and $\Mbar\satisfies\ZFCmm$.
  \item The class $\Mbar$ and the ultrapower map are definable in $W$.
  \item The \Los\ theorem fails for this ultrapower at the $\Sigma_1$-level.
  \item The Gaifman theorem fails for the ultrapower
      map $j:W\to \Mbar$, since it is
      $\Sigma_1$-elementary and cofinal, but not
      $\Sigma_2$-elementary.
 \end{enumerate}
\end{theorem}

\begin{proof}
We already observed in section \S\ref{Section.LevyCollapseInaccessible} that $W\satisfies \ZFCmm$ but not collection, and that $W^{\ltkappa}\of W$.
Let $\mu$ be any normal measure on $\kappa$ in $V$. Note that every
initial segment $\P_\gamma$ of the forcing is small relative to
$\kappa$. Thus, by the \Levy-Solovay theorem \cite{LevySolovay67} it
follows that $\kappa$ remains measurable in $V[G_\gamma]$, and
indeed, the filter $\mu_\gamma$ generated by $\mu$ in $V[G_\gamma]$
is a normal measure on $\kappa$ in $V[G_\gamma]$. Let $\mu^*$ be the
filter on $\kappa$ generated by $\mu$ in $W$, which is the same as
$\Union_\gamma \mu_\gamma$. Note that $\mu^*$ is not in $W$, since
it is not in any $V[G_\gamma]$, but it is definable over $W$ from
parameter $\mu$, since a set is in $\mu^*$ if and only if it covers
an element of $\mu$. In fact, $\mu^*$ is a $W$-normal measure on
$\kappa$, since every subset of $\kappa$ in $W$ is in some
$V[G_\gamma]$ and hence is measured by $\mu_\gamma$, and every
regressive function on $\kappa$ in $W$ is in some $V[G_\gamma]$ and
hence is constant on a $\mu_\gamma$-large set there. Moreover, the measure $\mu^*$ is countably complete in $V[G]$ since $W$
contains all its $\omega$-sequences from $V[G]$. We can construct in $V[G]$, since $W$ is definable there, the
ultrapower $\Ult(W,\mu^*)$ using functions on $\kappa$ in $W$.
Since $\mu^*$ is countably complete, it follows by the usual argument that $\Ult(W,\mu^*)$ is
well-founded, and so by taking the Mostowski collapse we obtain the ultrapower map $j:W\to \Mbar$ with $\Mbar$ transitive. In fact, we can both construct and collapse the ultrapower in $W$ itself, though this is not immediately obvious as Scott's trick that is crucial to this process may fail in the absence of power sets (see our remark after the proof). Note that the ultrapower map $j$ is $\Sigma_1$-elementary and cofinal, by our remarks in the introduction of this article.

Since the forcing $\P_\gamma$ for $\gamma\lt\kappa$ is small relative to~$\kappa$, it follows that the ultrapower map $j_0:V\to M$
 by $\mu$ in $V$ lifts uniquely to an elementary embedding
\hbox{$j_\gamma:V[G_\gamma]\to M[G_\gamma]$}, which necessarily equals the ultrapower map by
$\mu_\gamma$ in $V[G_\gamma]$.
In fact, whenever $\gamma<\delta$ are ordinals below $\kappa$, then $j_\delta\restrict V[G_\gamma]=j_\gamma$, since when we lift the ultrapower $j_\gamma$ to $V[G_\delta]$ we obtain by the smallness of the corresponding forcing precisely the ultrapower by $\mu_\delta$, which is the same as $j_\delta$. The union $\Union_{\gamma\lt\kappa} j_\gamma$ is thus a well-defined map. We claim that  $j=\Union_{\gamma\lt\kappa} j_\gamma$  and $\Mbar=\Union_{\gamma<\kappa}  M[G_\gamma] $ via the isomorphism $[f]_{\mu_\gamma}\mapsto [f]_{\mu^*}$ whenever $f:\kappa\to V[G_\gamma]$ is a function in $V[G_\gamma]$. This map is well-defined and $\in$-preserving since $[f]_{\mu_\gamma}=j_\gamma(f)(\kappa)=j_\delta(f)(\kappa)=[f]_{\mu_\delta}$ whenever
$\gamma<\delta$ and $[f]_{\mu_\gamma}\in M[G_\gamma]$. The map is clearly onto, and it follows that $[f]_{\mu_\gamma}=[f]_{\mu^*}$ whenever $[f]_{\mu_\gamma}\in M[G_\gamma]$. It follows that  $\Mbar\satisfies\ZFCmm$, since it is obtained
from $M$ by the \Levy\ collapse of $\kappa$ in the manner of section
\S\ref{Section.LevyCollapseInaccessible}.

Since $\kappa$ is an uncountable cardinal in $W$ and $\Mbar\of W$, it follows that $\kappa$ is an uncountable cardinal in $\Mbar$. Thus, although $\kappa=\omega_1^W$, it cannot be that $j(\kappa)=\omega_1^\Mbar$, since $\kappa$ itself is an uncountable cardinal strictly below $j(\kappa)$ in $\Mbar$. Thus $j$ is not elementary, and so the \Los\ theorem fails for the ultrapower $\Ult(W,\mu^*)$. Specifically, the map $j$ is not $\Sigma_2$-elementary, since the assertion ``$\kappa=\omega_1$" has complexity $\Pi_2$. Even without any prior knowledge about the precise structure of the ultrapower $\Ult(W,\mu^*)$ and its transitive collapse $\Mbar$, one can argue using the violation of collection that $j$ is not $\Sigma_2$-elementary, and that, indeed, \Los\ fails already for $\Sigma_1$-formulas. Recall that for all $\alpha<\kappa$, the model $W$ has surjections $f:\omega\surj\alpha$, but there cannot be a set in $W$ collecting a family of such functions. Now to see that \Los\ fails for $\Sigma_1$-formulas, observe that in the ultrapower there cannot exist a surjection from $[c_\omega]=\omega$ onto $[\id]=\kappa$, since otherwise, if $[g]:[c_\omega]\surj [\id]$ were such a surjection, then \hbox{$\{\alpha<\kappa\st g(\alpha) \text{ is a surjection from }\omega\text{ onto }\alpha\}$} would be a set in $\mu^*$ (by the \Los\  theorem for $\Delta_0$-formulas), and so there would be a collecting set in $W$. If $j$ were $\Sigma_2$-elementary, the ultrapower would have for all $\alpha<j(\kappa)$, a surjection from $\omega$ onto $\alpha$, and hence a surjection from $\omega$ onto $\kappa$.
\end{proof}

Before continuing, let us remark on some subtle issues concerning the extent to
which one can view the ultrapower of a $\ZFCmm$ or even a
$\ZFCm$ model as an internal construction inside such a
model. One issue is that even in the case that a measure
$\nu$ is definable in or perhaps even an element of a
$\ZFCm$ model $M$, then although one can define in $M$ the
fundamental relations $=_{\nu}$ and $\in_{\nu}$ used to
construct the ultrapower $\Ult(M,\nu)$, one seems unable in
general to find representing sets in a definable way for
the equivalence classes. Each equivalence class is, after
all, a proper class in $M$, an issue usually resolved in
the \ZFC\ context by means of Scott's trick, where one
restricts to the set of minimal-rank representatives in
each equivalence class; but Scott's trick doesn't succeed
in the $\ZFCmm$ or $\ZFCm$ contexts, because the collection
of minimal-rank representatives from a class may still not
be a set, when one lacks the power set axiom. Thus, one
seems to have difficulty performing the quotient operation,
defining the ultrapower quotient structure as a first-order
class model. Thus, if one wants to construct the ultrapower
internally, it seems that one may be forced always to deal
only with the pre-quotient structure, where one has only
the equivalence relation $=_{\nu}$ rather than a true
equality $=$ relation as in the quotient. A greater
difficulty is that even if one should be able definably to
find a representing set for each equivalence class, and
thereby have a quotient representation of the ultrapower as
a first-order class, one cannot necessarily perform the
Mostowski collapse, even when the ultrapower is
well-founded, because the $\in_\nu$ relation on those
classes is not necessarily set-like in $M$. In fact, even a
model of $\ZFCm$ containing a measurable cardinal is not
always able to construct the ultrapower and take its
Mostowski collapse. For an explicit example of this,
suppose that $\kappa$ is a measurable cardinal and fix an
elementary embedding $j:V\to M$ by a normal measure $\mu$
on $\kappa$ and a strong limit cardinal $\lambda$ of
cofinality $\kappa$. It follows that
$\lambda^+<j(\lambda)$, since $M$ is correct about ${}^\kappa\lambda$.
Consider $H_{\lambda^+}$, which is a model of $\ZFCm$
containing $\mu$ and all functions $f:\kappa\to
H_{\lambda^+}$. The collapsed ultrapower of $H_{\lambda^+}$
by $\mu$, therefore, is the same as the restriction
$j:H_{\lambda^+}\to H^M_{j(\lambda)^+}$, which is not a
subset of $H_{\lambda^+}$ since $H^M_{j(\lambda)^+}$
contains ordinals above $\lambda^+$. Thus, even though
$H_{\lambda^+}$ sees that $\mu$ is a measure on $\kappa$,
and is able to define the ultrapower relations and observe
correctly that the ultrapower is well-founded, it is not
able to perform the Mostowski collapse of this structure,
since it lacks sufficient ordinals to do so. One way to
describe the situation is that $H_{\lambda^+}$ does not
agree that a cardinal $\kappa$ is measurable (in the sense
of having a $\kappa$-complete nonprincipal ultrafilter on
$\kappa$) if and only if there is an ultrapower embedding
of the universe into a transitive class. So the equivalence
of these two characterizations of measurability is not
provable in $\ZFCm$.

Nevertheless, we can overcome these issues in the case of the
model $W$ and the measure $\mu^*$ that we construct in the proof of
theorem \ref{Section.LevyCollapseMeasurable}. In particular, let us
argue that both the Mostowski collapse $\Mbar$ of $\Ult(W,\mu^*)$
and the corresponding ultrapower map $j:W\to\Mbar$ of that proof are
definable classes in $W$. As we observed in the proof, for any $\gamma<\kappa$, the
 Mostowski collapse of $[f]$ in $\Ult(W,\mu^*)$ is the same as the
Mostowski collapse of $[f]$ in $\Ult(V[G_\gamma],\mu_\gamma)$. It follows that
 every function $g:\kappa\to V[G_\gamma]$ in $W$ is equivalent on
a set in $\mu$ to a function $g':\kappa\to V[G_\gamma]$ with $g'\in
V[G_\gamma]$.
This means that we can compute in $W$ the image of the Mostowski
collapse of $[f]$ in $\Ult(W,\mu^*)$ by performing the collapse
inside $V_\theta[G_\gamma]$ for large enough $\theta$. Indeed, all
that is needed is a sufficiently large transitive set $A$, such that
$f\in A$ and for any function $g\in W$ having $g\in_{\mu^*} h\in A$
for some function $h$, then there is $g'\in A$ with $g=_{\mu^*} g'$.
In this case, one shows that the Mostowski collapse of $[f]$ in
$\Ult(W,\mu^*)$ is the same as the collapse of $[f]$ in
$\<A,\in_{\mu^*}>/{=_{\mu^*}}$. Since there are abundant such $A$ in
$W$, such as $A=V_\theta[G_\gamma]$ for any sufficiently large
$\theta$, and they all give rise to the same value for the Mostowski
collapse of $[f]$, it follows that in $W$ we may definably associate
any function $f$ to its image under the Mostowski collapse of
$\Ult(W,\mu^*)$. Thus, $\Mbar$ is a definable class in $W$, and by
considering the constant functions, we may also thereby define the
ultrapower map $j:W\to \Mbar$.

Let us now turn to the \Levy\ hierarchy of formulas in the language of set theory, where $\Sigma_n$-formulas and $\Pi_n$-formulas are defined as usual in a purely syntactical way. Recall that if a formula $\varphi$ is obtained by bounded quantification over some $\Sigma_n$-formula, then there exists another $\Sigma_n$-formula $\varphi'$ that is provably equivalent in $\ZFC$ to $\varphi$. The standard proof of this fact uses repeated applications of collection and the pairing axiom, and so it  follows that the equivalence between $\varphi$ and  $\varphi'$ can also be proved in $\ZFCm$. Corollary~\ref{Corollary:Sigma_1NotClosed}
shows that the use of collection is essential to this argument, and one cannot obtain in some other way a formula $\varphi'$ of reduced complexity that is provably equivalent in $\ZFCmm$ to $\varphi$.

\begin{corollary}\label{Corollary:Sigma_1NotClosed}
Relative to a measurable cardinal, the collection of formulas that are provably equivalent in $\ZFCmm$ to a $\Sigma_1$-formula or a $\Pi_1$-formula is not closed under bounded quantification.\end{corollary}

\begin{proof}
Let $\varphi(x)$ be the formula asserting that all elements of $x$ are countable, meaning that for each nonempty $y\in x$ there exists a surjection from $\omega$ onto $y$.  The formula $\varphi(x)$ is clearly obtained by bounded universal quantification over a $\Sigma_1$-formula, but we will show that it is not provably equivalent in $\ZFCmm$ to any $\Sigma_1$-formula or $\Pi_1$-formula. Consider the $\Sigma_1$-elementary embedding $j:W\to \Mbar$ from theorem~\ref{Theorem.LevyCollapseMeasurable} and note that the model $W$ satisfies $\varphi(\kappa)$ since it has collapsing functions for all ordinals below $\kappa$. If  there would be some $\Sigma_1$-formula $\varphi '(x)$ such that $\ZFCmm$ proves \hbox{$\forall x(\varphi(x)\leftrightarrow \varphi'(x))$}, then it would follow that $W\satisfies \varphi'(\kappa)$, and so $\Mbar\satisfies \varphi'(j(\kappa))$ by the $\Sigma_1$-elementarity of $j:W\to\Mbar$. Since $\Mbar\satisfies\ZFCmm$ , it would then follow that $M\satisfies \varphi(j(\kappa))$, which means that $\kappa$ would be countable in $M$, which is not the case. The same argument shows that there cannot be a $\Pi_1$-formula provably equivalent in $\ZFCmm$ to $\varphi$. \end{proof}

We will improve this result by avoiding the need for a measurable cardinal in section~\S\ref{Section.NoLargeCardinals}. Also, in subsequent sections, we will obtain counterexamples to
the \Los\ theorem involving ultrafilters that exist inside
the model, including ultrafilters on $\omega$ when
$P(\omega)$ exists, as well as other violations of the \Los\ and Gaifman theorems, and other counterexamples that do not require any large cardinals.

\subsection{A cofinal restriction of an elementary
embedding}\label{Section.CofinalRestrictionNotElementary}

In the previous section we proved that the Gaifman theorem can fail
for $\ZFCmm$ models, in the sense that there can be $j:M\to N$ for
transitive $\ZFCmm$ models $M$ and $N$, which is
$\Sigma_1$-elementary and cofinal, but not $\Sigma_2$-elementary. We
would like now to describe a dual situation, where one has a fully
elementary but non-cofinal embedding $j:M\to N$ of transitive $\ZFCmm$ models, whose canonical cofinal restriction $j:M\to \Union j\image M$ is not elementary. Note that since $M$ and $N$ are
transitive, then so is $\Union j\image M$, since it is a union of transitive sets (using replacement, one shows that $M$ satisfies that every set has a transitive closure). The restriction $j:M\to \Union j\image M$ is a $\Sigma_1$-elementary embedding, since it is clearly $\Delta_0$-elementary and also cofinal, but theorem~\ref{Theorem.CofinalRestrictionNotElementary} shows that it need not be $\Sigma_2$-elementary if collection fails in $M$. This stands in contrast to the situation when $M\satisfies \ZFCm$ and $j:M\to N$ is fully elementary, since one can then argue that $\Union j\image M \elesub N$ by the Tarski-Vaught test and conclude that $j:M\to \Union j\image M$ is fully elementary.

\begin{theorem}\label{Theorem.CofinalRestrictionNotElementary}
Relative to a measurable cardinal, it is consistent that there are transitive models $M$ and $N$ of $\ZFCmm$ with a fully elementary
embedding \hbox{$j:M\to N$}, whose cofinal restriction
$j:M\to \Union j\image M$ is an embedding of $\ZFCmm$ models that is not $\Sigma_2$-elementary.
\end{theorem}

\begin{proof}
Let us suppose that $\kappa$ is a measurable cardinal in $V$.
Let $V[G]$ be the \Levy\ collapse of $\kappa$, meaning that
$G\of\P=\Coll(\omega,\ltkappa)$ is $V$-generic, and let
$W=\Union_{\gamma<\kappa} V[G_\gamma]$, as in section
\S\ref{Section.LevyCollapseInaccessible}, so that we know
$W\satisfies\ZFCmm$. Let $j:V\to M$ be the ultrapower by a normal
measure $\mu$ on $\kappa$ in $V$, and consider the forcing
$j(\P)=\Coll(\omega,{\lt}j(\kappa))$, which can be factored as
$j(\P)\cong\P\times\P_{\kappa,j(\kappa)}$, where
$\P_{\kappa,j(\kappa)}$ is the part of the forcing collapsing
cardinals in the interval $[\kappa,j(\kappa))$. Suppose that
$H\of\P_{\kappa,j(\kappa)}$ is $V[G]$-generic. Since $G\times H$ is
$V$-generic for $j(\P)$ and $j\image G=G$, it follows that $j$ lifts
to the fully elementary $j:V[G]\to M[j(G)]$ in $V[G][H]$, where $j(G)=G\times H$. Since $W$ is definable in $V[G]$,
using Laver's result~\cite{Laver2007:CertainVeryLargeCardinalsNotCreated} on the uniform definability of the ground model in the forcing extension from the generic filter and a poset-dependent parameter, we may restrict $j$ to $W$ and obtain in $V[G][H]$ the fully elementary embedding $j\restrict W: W\to j(W)$, where $j(W)=\Union_{\gamma\lt j(\kappa)} M[j(G)_\gamma]$ by the uniformity of Laver's definition.
The cofinal restriction $j\restrict W:W\to \Union j\image W$ is $\Sigma_1$-elementary, but it is not fully elementary, since all ordinals below $\kappa$ are countable in $W$, but the ordinal $\kappa$ below $j(\kappa)$ cannot be countable in $\Union j\image W$, since it is easy to see that $\Union j\image W\of \Union_{\gamma<\kappa}M[j(G)_\gamma]=\Union_{\gamma<\kappa}M[G_\gamma]$ and no $M[G_\gamma]$ can collapse cardinals above $\gamma$. The assertion ``$\kappa=\omega_1$" has complexity $\Pi_2$, and we observe more precisely that $j:W\to \Union j\image W$ is not $\Sigma_2$-elementary. The embedding $j\restrict W:W\to j(W)$ is thus not cofinal, but to see this directly, note that, for example, the reals added by $j(G)_\kappa$ are in $j(W)$ but not in $\Union j\image W$.

Indeed, $\Union j\image W= \Union_{\gamma<\kappa}M[G_\gamma]$, since $x\in M[G_\gamma]$ implies that $x=j(f)(\kappa)$ for some function $f\in V[G_\gamma]$, as  the restriction of $j$ to $V[G_\gamma]$ is a lift of the ultrapower map $j$ and hence an ultrapower map as well, and consequently $x\in j(\ran f)\of \Union j\image W$. It follows that $\Union j\image W$ satisfies $\ZFCmm$. Note finally that the map \hbox{$j\restrict W:W\to
\Union j\image W$} is precisely the
ultrapower of $W$ by $\mu^*$ as described in section
\S\ref{Section.LevyCollapseMeasurable} by uniqueness of the lift of $j$ to each $V[G_\gamma]$.
\end{proof}

Theorem~\ref{Theorem.CofinalRestrictionNotElementary} can be used to show that standard arguments from seed theory\footnote{For a review of elementary seed theory, see for instance~\cite{Hamkins97:Seeds}.}, permissible in the context of $\ZFCm$ models, can fail in the context of $\ZFCmm$ models.

\begin{corollary}\label{Corollary.SeedHull}
Relative to a measurable cardinal, it is consistent that there are elementary embeddings \hbox{$j:M\to N$} of transitive models of $\ZFCmm$ and sets \hbox{$S\of \Union j\image M$} such that the seed hull $\X_S=\{j(f)(s)\st s\in [S]^\ltomega, f\in M\}$ is not a $\Sigma_1$-elementary substructure of $N$, and such that the restriction $j:M\to\X_S$ is not $\Sigma_2$-elementary.
\end{corollary}

\begin{proof}
Consider the elementary embedding $j\restrict W:W\to j(W)$ of the proof of theorem~\ref{Theorem.CofinalRestrictionNotElementary} and let $X_{\{\kappa\}}=\{j(f)(\kappa)\st f:\kappa\to W, f\in W\}$ be the seed hull of $\{\kappa\}$ via $j$. Observe that $X_{\{\kappa\}}=\Union_{\gamma<\kappa}M[G_\gamma]$. The seed hull $X_{\{\kappa\}}$ is not a $\Sigma_1$-elementary substructure of $j(W)$, since if it were, then $\X_{\{\kappa\}}$ would have a surjection from $\omega$ onto $\kappa$, but we know that this is not the case, since no $M[G_\gamma]$ collapses $\kappa$ if $\gamma<\kappa$.
\end{proof}

\subsection{Violating \Los\ with a measurable cardinal inside the
model}\label{Section.ViolatingLosWithMeasurableInside} We
shall now modify the construction of section
\S\ref{Section.LevyCollapseMeasurable} in order to arrive at
a more definitive violation of the \Los\ theorem for
$\ZFCmm$ models, by producing a model $W\satisfies\ZFCmm$
in which there is a measurable cardinal $\kappa$ whose
powerset $P(\kappa)$ exists in $W$ and for which there is a
$\kappa$-complete normal ultrafilter $\mu$ on $\kappa$ in
$W$, whose ultrapower $\Ult(W,\mu)$ is well-founded and can
be constructed and collapsed inside $W$, but the ultrapower
map $j:W\to \Ult(W,\mu)$ is not elementary.

\begin{theorem}\label{Theorem.LevyCollapseAboveMeasurable}
Relative to a measurable cardinal, it is consistent with $\ZFCmm$ that there is a
measurable cardinal $\kappa$ for which $P(\kappa)$ exists
and there is a $\kappa$-complete normal measure on
$\kappa$, whose ultrapower is a well-founded model of $\ZFCmm$ that can be constructed and collapsed, but for which
the \Los\ theorem fails at the $\Sigma_1$-level and the ultrapower map is not
$\Sigma_2$-elementary.
\end{theorem}

\begin{proof}
The idea is to combine the methods of sections
\S\ref{Section.LevyCollapseAleph_omega} and
\S\ref{Section.LevyCollapseMeasurable}, but performing all
the collapse forcing only above the measurable cardinal, so
that this cardinal and its measure will be preserved to the
extension and to the resulting model $W$ of $\ZFCmm$.
Suppose that $\kappa$ is a measurable cardinal in
$V\satisfies\ZFC+\GCH$, and let
$\lambda=\kappa^{\plus^\kappa}=\aleph_{\kappa+\kappa}$,
the $\kappa^\th$ cardinal successor to $\kappa$. (The \GCH\
assumption can be relaxed simply by using
$\beth_{\kappa+\kappa}$ in place of $\lambda$.)  Let $\P$ be the \Levy\ collapse of $\lambda$, meaning the bounded-support product $\Coll(\kappa^+,{\lt}\lambda)=\prod_{\alpha<\kappa}\Coll(\kappa^+,\aleph_{\kappa+\alpha})$.
Although the forcing $\Coll(\kappa^\plus,\aleph_{\kappa+\alpha})$ on any
coordinate $\aleph_{\kappa+\alpha}$ is $\leqkappa$-closed, the $\P$ forcing
is not $\leqkappa$-closed because of the bounded-support
requirement, and in the limit $\P$ will actually collapse
$\lambda$ to $\kappa$, since the function mapping
$\alpha<\kappa$ to the first ordinal used on coordinate
$\aleph_{\kappa+\alpha}$ will map $\kappa$ onto $\lambda$.
But every initial segment of the forcing
$\P_\gamma=\prod_{\alpha\leq\gamma}\Coll(\kappa^\plus,\aleph_{\kappa+\alpha})$,
for $\gamma\lt\kappa$, is $\leqkappa$-closed and therefore
adds no subsets to $\kappa$.

Suppose that $G\of\P$ is $V$-generic for this forcing, and let
$G_\gamma=G\intersect\P_\gamma$. Our intended model is
$W=\Union_{\gamma<\kappa}V[G_\gamma]$, which satisfies $\ZFCmm $, by arguments analogous to those previously given, but not collection. To see that collection fails in $W$, observe that  for each $\alpha<\kappa$ there exists a surjective function $f:\kappa^+\surj\aleph_{\kappa+\alpha}^V$ in $W$, but there cannot be a set in $W$ collecting a family of such functions.

Since all $\P_\gamma$ are $\leqkappa$-closed, we have that
$P(\kappa)^V=P(\kappa)^{V[G_\gamma]}=P(\kappa)^W$. Thus,
$P(\kappa)$ exists unchanged in $W$. Furthermore, if $\mu$
is any normal measure on $\kappa$ in $V$, then $\mu$
continues to be a normal measure on $\kappa$ in $W$. Note
that for every $x\in W$, the collection of all $f:\kappa\to
x$ forms a set in $W$ since if $x\in V[G_\gamma]$, then no
new functions from $\kappa$ to $x$ are added by subsequent
collapses. Therefore $W$ can construct the ultrapower
$\Ult(W,\mu)$ as the union of the ultrapowers $\Ult(x,\mu)$
over all transitive sets $x\in W$ and take its Mostowski
collapse (this requires replacement) to a model $\Mbar$ and
let $j:W\to \Mbar$ be the resulting embedding.

We shall now observe that \Los\ already fails at the $\Sigma_1$-level for this ultrapower and that the map $j:W\to \Mbar$ is not $\Sigma_2$-elementary, using the violation of collection described above. Let $h$ be a function on $\kappa$ in $V\of W$ such that $h(\alpha)=\aleph_{\kappa+\alpha}^V$, and recall that for all $\alpha<\kappa$, the model $W$ has surjections from $\kappa^+$ onto $h(\alpha)$ but no collecting set of such functions. It follows that there cannot be a surjection \hbox{$[f]:j(\kappa^+)\surj [h]$} in the ultrapower, since otherwise the set $\{\alpha<\kappa\st f:\kappa^+\surj h(\alpha)\}$ would be an element of $\mu$ (by \Los\ for $\Delta_0$-formulas), yielding a collecting set, and thus \Los\ fails at the $\Sigma_1$-level. It also follows that the ultrapower map $j$ is not elementary for the $\Pi_2$-formula $\forall \alpha{<}\kappa\,\exists f\, f:\kappa^+\surj h(\alpha)$ using $\kappa^+$ and $h$ as parameters, since otherwise the ultrapower would have a surjection from $j(\kappa^+)$ onto $j(h)(\kappa)=[h]$.

It remains to argue that $\Mbar$ is nevertheless a model of $\ZFCmm$. Let $j:V\to M$ be the ultrapower embedding by $\mu$ in $V$.
Since the forcing $\P_\gamma$ is $\leqkappa$-closed, it follows that $j$
lifts uniquely to $j_\gamma:V[G_\gamma]\to M[j(G_\gamma)]$, where
$j(G_\gamma)$ is the filter in $j(\P_\gamma)$ generated by $j\image
G_\gamma$. (This is generic since every open dense set $D\of
j(\P_\gamma)$ is $j(\vec D)(\kappa)$ for some $\vec
D=\<D_\alpha\st\alpha<\kappa>$ with $D_\alpha\of\P_\gamma$ open
dense, and by $\leqkappa$-closure it follows that
$\Dbar=\Intersect_\alpha D_\alpha$ remains dense and has
$j(\Dbar)\of D$; since $G_\gamma$ meets $\Dbar$, it follows that
$j\image G_\gamma$ meets $j(\Dbar)$ and hence $D$, as desired.) We
claim that $\Mbar=\Union_{\gamma\lt\lambda} M[j(G_\gamma)]$ and
$j=\Union_{\gamma\lt\kappa} j_\gamma$ by the map that associates
$[f]_\mu\in\Ult(W,\mu)$ to $[f]_\mu$ in any $\Ult(V[G_\gamma],\mu)$
for which $f\in V[G_\gamma]$. This association is well-defined and $\in$-preserving, since if $f\in V[G_\gamma]$ and $\gamma<\delta<\kappa$, then $j_\gamma(f)(\kappa)=[f]_\mu$ as computed in  $\Ult(V[G_\gamma],\mu)$, which is the same as $j_\delta(f)(\kappa)=[f]_\mu$ as computed in $\Ult(V[G_\delta],\mu)$, by the uniqueness of the lifted embeddings. The association is onto and hence an
isomorphism, and it commutes with the ultrapower maps. Since $\Mbar=\Union_{\gamma\lt\lambda} M[j(G_\gamma)]$, one can verify that $\Mbar\satisfies\ZFCmm$ just as we previously verified $W\satisfies\ZFCmm$, except that now we are collapsing all cardinals below $\aleph_{j(\kappa)+\kappa}^M$ to $j(\kappa^+)$ over $M$.
\end{proof}

\subsection{Violating the \Los\  and Gaifman theorems without large cardinals.}\label{Section.NoLargeCardinals} Analogous arguments as in theorems~\ref{Theorem.LevyCollapseMeasurable},~\ref{Theorem.CofinalRestrictionNotElementary} and~\ref{Theorem.LevyCollapseAboveMeasurable} and corollaries~\ref{Corollary:Sigma_1NotClosed} and~\ref{Corollary.SeedHull} work in the context of elementary embeddings characterizing smaller large cardinals.  For instance, if $\kappa$ is weakly compact, we can find a transitive set $M\satisfies\ZFC$  of size $\kappa$ with $V_\kappa\in M$ and an ultrapower map $j:M\to N$ by an $M$-normal measure on $\kappa$ such that $N$ is transitive. The poset $\Coll(\omega,\ltkappa)$ is an element of $M$ as it is definable over $V_\kappa$, and so in $V[G]$, the \Levy\ collapse of $\kappa$, we may form the forcing extension $M[G]$ and let $W=\Union_{\gamma<\kappa} M[G_\gamma]$, which satisfies $\ZFCmm$, but not collection. The rest of the constructions now proceed identically to those previously given.

In fact, these arguments can be modified to work with elementary embeddings of transitive models whose existence follows directly from $\ZFC$, thereby avoiding any need for large cardinals.

\begin{theorem}\label{Theorem.LevyCollapseAlephOmega1ForHauser}
There are transitive set models $M$ and $N$ of $\ZFCmm$ with a fully elementary embedding $j:M\to N$, 
whose cofinal restriction \hbox{$j:M\to \Union j\image M$} is not $\Sigma_2$-elementary.
\end{theorem}

\begin{proof}
We start by constructing an appropriate elementary embedding $j:M\to N$ with critical point $\omega_1^M$. First, we may assume without loss of generality that GCH holds, by passing if necessary to an inner model. Let $\theta>2^{\aleph_{\omega_1}}$ be any regular cardinal. Let $X$ be a countable elementary substructure of $H_\theta$ with Mostowski collapse $\pi_X:X\to M$, and observe that $\omega_1^M$ exists in $M$, and that $M\satisfies \ZFCm$. Using \CH, let $Y\supseteq X$ be an elementary substructure of $H_\theta$ of size $\omega_1$ with $Y^\omega\of Y$. Let $\pi_Y:Y\to N$ be the Mostowski collapse of $Y$, and observe that $N\satisfies\ZFCm$ and $N^\omega\of N$. The composition map $j=\pi_Y\circ \pi_X^{-1}$ is then an elementary embedding $j:M\to N$ with critical point $\omega_1^M$, that is $\omega_1^M<j(\omega_1^M)=\omega_1^N=\omega_1$, so that $\omega_1^M$ is a countable ordinal in $N$.

Working inside the model $M$, note that $P(\aleph_{\omega_1})$ exists and we may thus consider the \Levy\ collapse of cardinals below $\aleph_{\omega_1}$ to $\omega_1$, meaning the bounded-support product $\P=\Coll(\omega_1,{\lt}\aleph_{\omega_1})=\prod_{\beta<\omega_1}\Coll(\omega_1,\aleph_\beta)$. The forcing $\P$ is countably closed in $M$, but not $\leqomega_1$-closed, and forcing with $\P$ will collapse $\aleph_{\omega_1}$ to $\omega_1$; for example, the function mapping $\alpha<\omega_1$ to the first ordinal mentioned on coordinate $\aleph_\alpha$ will map $\omega_1$ onto $\aleph_{\omega_1}$.

Since $\P$ is the \Levy\ collapse of $\aleph_{\omega_1}^M$, collapsing to $\omega_1^M$, as computed in the model $M$, it follows by elementarity that $j(\P)=\Coll(\omega_1,{\lt}\aleph_{\omega_1})^N$ is the \Levy\ collapse of $\aleph_{\omega_1}^N$, collapsing to $\omega_1^N$, as computed in $N$. Note that $j(\P)$ is countably closed, since it is countably closed in $N$ and $N^\omega\of N$. We aim to lift the elementary embedding $j:M\to N$ to forcing extensions of $M$ and $N$ by $\P$ and $j(\P)$, respectively, using generic filters that exist in $V$.
Since $M$ is countable, there  exists an $M$-generic filter $G\of \P$ that is generated by a countable descending sequence $\{p_n\st n\in \omega\}$ of conditions  in $\P$. By closure of $j(\P)$, we can find a condition $q\in j(\P)$ that is below all the $j(p_n)$, and since $N$ has size $\omega_1$, we can  build an $N$-generic filter $j(G)\of j(\P)$ having $q$ as an element. It follows that $j\image G\of j(G)$, and so we can lift the embedding in $V$ to $j:M[G]\to N[j(G)]$.

For each $\gamma<\omega_1^M$, let $G_\gamma=G\intersect \P_\gamma$. Our intended model is $W=\Union_{\gamma<\omega_1^M} M[G_\gamma]$, which satisfies $\ZFCmm$, by arguments analogous to those previously given, but not collection. To see that collection fails in $W$, observe that for every $\alpha<\aleph_{\omega_1}^M$, the model $W$ has a surjective function $f:\omega_1^M\surj \alpha$, but no set containing a family of such functions by the chain condition of $\P_\gamma$ as the \GCH\ holds in $M$, and thus there cannot be a  set in $W$ collecting a family of such functions.

We may assume without loss of generality that $W$ is a definable class in $M[G]$, by starting the construction if necessary\footnote{The proof of the ground model definability theorem~\cite{Laver2007:CertainVeryLargeCardinalsNotCreated}, based on the methods of \cite{Hamkins2003:ExtensionsWithApproximationAndCoverProperties}, makes essential use of the power set axiom, and one of the main results in \cite{GitmanJohnstone2014:GroundModels} shows that ground models of $\ZFCm$ and $\ZFCmm$ models need not be definable in their forcing extensions.} in an inner model such as $L$. We may therefore restrict $j$ to $W$ and obtain, as in theorem~\ref{Theorem.CofinalRestrictionNotElementary}, a fully elementary embedding $j\restrict W: W\to j(W)$, where $j(W)=\Union_{\gamma<\omega_1}N[j(G)_\gamma)]$. The canonical cofinal restriction $j\restrict W:W\to \Union j\image W$ is $\Sigma_1$-elementary, but it is not fully elementary, since for each $\alpha<\aleph_{\omega_1}^M$ there is a function in $W$ collapsing $\alpha$ to $\omega_1^M$, but there is no function in  $\Union j\image W\of \Union_{\gamma<{\omega_1^M}}N[j(G)_\gamma]$ that collapses the cardinal $\aleph_{\omega_1^M}^N$, which is strictly below $j(\aleph_{\omega_1}^M)=\aleph_{\omega_1}^N$, since no $N[j(G)_\gamma]$ collapses cardinals  above $\aleph_\gamma^N$. Altogether, we have observed that there are no cardinals between $\omega_1^M$ and $\aleph_{\omega_1}^M$ in $W$, but there are cardinals in $\Union j\image W$ between $j(\omega_1^M)=\omega_1^N$ and $j(\aleph_{\omega_1}^M)=\aleph_{\omega_1}^N$, since $\aleph_{\omega_1^M}^N$  itself is such a cardinal. Since the assertion ``$\lambda=\delta^\plus$" where $\lambda=\aleph_{\omega_1}^M$ and $\delta=\omega_1^M$ has complexity $\Pi_2$, it follows that the cofinal restriction $j\restrict W:W\to \Union j\image W$ is not $\Sigma_2$-elementary.
\end{proof}

Using the map $j:W\to \Union j\image W$ from theorem~\ref{Theorem.LevyCollapseAlephOmega1ForHauser}, it is easy to obtain
the result of corollary \ref{Corollary:Sigma_1NotClosed}, this time without large cardinals.

Theorem~\ref{Theorem.LevyCollapseAlephOmega1ForHauser} can also be used to provide failures of the elementarity of seed hulls, but in addition to the violations of corollary~\ref{Corollary.SeedHull}, we now obtain for a single elementary embedding $j:M\to N$ uncountably many distinct seed hulls $\X_S$ that are not elementary substructures of $N$ and that are not generated by a single seed. Indeed, consider the elementary embedding $j\restrict W:W\to j(W)$ from the proof of theorem~\ref{Theorem.LevyCollapseAlephOmega1ForHauser}. Let $S\of \Union j\image W$ be any seed set containing $\aleph^N_{\omega_1^M}$ as a subset and note that $\X_S\of \Union j\image W$.  The seed hull $\X_S$ cannot be generated by a single seed, since $W$ is countable, but $\X_S$ is uncountable. If $\X_S$ were $\Sigma_1$-elementary in $j(W)$, then $\X_S$ and hence $\Union j\image W$ would contain a surjection from $\omega_2^M$ onto $\aleph^N_{\omega_1^M}$, but we argued already in theorem~\ref{Theorem.LevyCollapseAlephOmega1ForHauser} that this is not the case.
Observe that there are uncountably many distinct seed hulls: Since $N$ is closed under countable sequences, it follows that $N[j(G)]^\omega\of N[j(G)]$ and consequently $j(W)^\omega\of j(W)$, and so $j\restrict W$ is an element of $j(W)$. The model $j(W)$ can construct seed hulls for seed sets $S\in j(W)$ of different cardinalities in $j(W)$, and the corresponding seed hulls $\X_S$ are thus all distinct. Lastly, the map $j:W\to \X_S$ is a $\Sigma_1$-elementary cofinal map that is not $\Sigma_2$-elementary, and the same is true for the map $\pi\circ j:W\to \ran(\pi)$ where  $\pi:\X_S\to \ran(\pi)$ denotes the transitive collapse of $\X_S$. These two maps provide therefore counterexamples to the Gaifman theorem in the $\ZFCmm$ context for maps that are not ultrapower maps.

Essentially the same arguments as in theorem~\ref{Theorem.LevyCollapseAlephOmega1ForHauser} allow us to obtain a failure of the \Los\ theorem, without any need for large cardinals.

\begin{theorem}\label{Theorem.LevyCollapseAlephOmega1ForUltrapower}
There is a transitive set $M\satisfies \ZFCmm$  and an $M$-normal measure $\mu\of P(\omega_1)^M$ for which the ultrapower of $M$ by $\mu$ is a well-founded model of $\ZFCmm$, but the ultrapower map  is not $\Sigma_2$-elementary, and \Los\ fails at the $\Sigma_1$-level.
\end{theorem}

\begin{proof}
Following the proof of theorem~\ref{Theorem.LevyCollapseAlephOmega1ForHauser}, we obtain an elementary embedding $j:M\to N$ with critical point $\omega_1^M$, where $M$ is the Mostowski collapse of some countable $X\elesub H_\theta$ for a regular cardinal $\theta>2^{\aleph_{\omega_1}}$ and $N$ is the Mostowski collapse of some $Y\elesub H_\theta$ with $X\of Y$. If $\mu\of P(\omega_1)^M$ is the $M$-normal measure on $\omega_1^M$ that is obtained from $j:M\to N$ by using $\omega_1^M$ as a seed, then the ultrapower $\Ult(M,\mu)$ of $M$, using functions on $\omega_1^M$ in $M$, is well-founded, and so we may assume without loss of generality that $j:M\to N$ is the ultrapower by $\mu$. It follows, in particular, that $N$ is countable.

Again, we assume that \GCH\ holds. As in theorem~\ref{Theorem.LevyCollapseAlephOmega1ForHauser}, we consider the bounded-support product $\P=\Coll(\omega_2,{\lt}\aleph_{\omega_1})^M$, and since $M$ is countable, choose in $V$ an $M$-generic filter $G\of\P$ and let $G_\gamma=G\intersect P_\gamma$ for each $\gamma<\omega_1^M$. Our intended model is again $W=\Union_{\gamma<\omega_1^M} M[G_\gamma]$.
Since each initial forcing $\P_\gamma$ is $\leqomega_1$-closed in $M$, it follows, as in theorem~\ref{Theorem.LevyCollapseAboveMeasurable}, that $j:M\to N$ lifts uniquely to $j_\gamma:M[G_\gamma]\to N[j(G_\gamma)]$, where $j(G_\gamma)$ is the $N$-generic filter in $j(\P)$ generated by $j\image G_\gamma$, and the map $j_\gamma$ is necessarily equal to the ultrapower map of $M[G_\gamma]$ by $\mu$. It follows that the ultrapower $\Ult(W,\mu)$ of $W$, using functions on $\omega_1^M$ in $W$, is  well-founded, since if $j:W\to \Nbar$ is the ultrapower map, then $\Nbar=\Union_{\gamma<\omega_1^M}N[j(G_\gamma)]$ and $j=\Union_{\gamma<\omega_1^M}j_\gamma$ via the isomorphism that associates $[f]_\mu\in \Ult(W,\mu)$ to $[f]_\mu$ in any $\Ult(M[G_\gamma],\mu)$ for which $f\in M[G_\gamma]$. Thus, $\Nbar$ satisfies $\ZFCmm$. An argument analogous to that in the proof of theorem~\ref{Theorem.LevyCollapseAlephOmega1ForHauser} uses the violation of collection to show that \Los\ fails at the $\Sigma_1$-level and the ultrapower map $j:W\to \Nbar$ is not $\Sigma_2$-elementary. Thus, $j$ witnesses the failure of both the \Los\ and Gaifman theorems.
\end{proof}

\subsection{Violating \Los\ for ultrapowers on $\omega$}
In this final section, we aim to produce a counterexample
to the \Los\ and Gaifman theorems for $\ZFCmm$ models for ultrapowers by ultrafilters on $\omega$.

\begin{theorem}\label{Theorem.LevyCollapseAboveAlephomega}
In a forcing extension $V[G]$ by the \Levy\ collapse $G\of\Coll(\aleph_1,{\lt}\aleph_\omega)$  there is a transitive class inner model $W\satisfies\ZFCmm$  in which $P(\omega)$ exists and there
is an ultrafilter $\mu$ on $\omega$, such that the
ultrapower map $j:W\to\Ult(W,\mu)$ is definable in $W$, but is not $\Sigma_2$-elementary and, indeed, \Los\ fails at the $\Sigma_1$-level.
\end{theorem}

\begin{proof}
First, we may assume without loss of generality that $\GCH$ holds in $V$, by passing if necessary to an inner model such as $L$. Suppose that
$G\of\Coll(\aleph_1,{\lt}\aleph_\omega)$ is $V$-generic for
the \Levy\ collapse up to $\aleph_\omega$, that is, the
finite-support product
$\P=\prod_{n\lt\omega}\Coll(\aleph_1,\aleph_n)$.
Let $W=\Union_n V[G_n]$, where $G_n=G\intersect\P_n$ and
$\P_n=\prod_{k\leq n}\Coll(\aleph_1,\aleph_k)$. The model $W$ satisfies $\ZFCmm$, but not collection.
Since each $\P_n$ is countably closed and therefore does
not add subsets to $\omega$, we have
$P(\omega)^V=P(\omega)^{V[G_n]}=P(\omega)^W$, meaning that
$P(\omega)$ exists in $W$. If $\mu$ is any nonprincipal ultrafilter
on $\omega$ in $V$, then it continues to have this property in $W$.
Let $j:W\to\Ult(W,\mu)$ be the ultrapower of $W$ by $\mu$, as
computed by $W$. The model $W$ is able to take the ultrapower since
for every $x\in W$, the collection of all $f:\omega\to x$ forms a
set in $W$ since if $x\in V[G_n]$, then no new functions from
$\kappa$ to $x$ are added by the subsequent collapses. The failure of \Los\ at the $\Sigma_1$-level and the failure of elementarity for the ultrapower map at the $\Sigma_2$-level follow from the violation of collection as in the previous arguments. The Gaifman theorem fails for this embedding as well, as it is $\Sigma_1$-elementary and cofinal.
\end{proof}
\section{Some final remarks}

It is clear that the method of proof of our theorems is
both flexible and general and would easily adapt to the use of other
cardinals than the ones we have used. For example, in section
\ref{Section.LevyCollapseAleph_omega} we could have collapsed to
different cardinals, as in theorems
\ref{Theorem.LevyCollapseAboveMeasurable} and
\ref{Theorem.LevyCollapseAboveAlephomega}, in order to show that
other successor cardinals can be singular, or that several cardinals
might be singular. Indeed,
\cite{Zarach1996:ReplacmentDoesNotImplyCollection} provides a
general framework for producing models of $\ZFCmm$, involving
forcing over the weak product of $\omega$-many copies of a family of
forcing notions, and forming the desired model via increasing finite
portions of the product. The proof that his general framework
succeeds amounts essentially to the particular arguments we made for
our iterations, where we swapped two copies of the forcing $\P_n$ or
$\P_\gamma$ and appealed to weak homogeneity. Although we could
have appealed to Zarach's general framework, we chose simply to give
a direct argument in each case in order to achieve a self-contained
presentation. But we refer any reader interested in producing even
more badly behaved models of $\ZFCmm$ to consult Zarach's general
framework. The violations of \Los\ that we produced in our example can also be summarized in a general framework, as may be deduced from the examples we provided in the paper.
For our own part, we shall from now on prefer the models
of $\ZFCm$ over $\ZFCmm$, and we take the results of this article to
show definitively that $\ZFCmm$ is the wrong theory for most
applications of set theory without power set.

\bibliographystyle{alpha}
\bibliography{HamkinsBiblio,MathBiblio,zfcminus}

\end{document}